\documentclass[10pt, a4paper]{article}
\usepackage{}
\usepackage{amsfonts}
\usepackage{mathrsfs}
\usepackage{amsthm,amssymb}
\usepackage{amsmath}
\usepackage{latexsym,bm}
\usepackage{graphicx}
\usepackage{subfigure}
\usepackage{marvosym}
\usepackage{a4wide}
\usepackage{authblk}
\usepackage{mathtools}
\usepackage{hyperref}
\hypersetup{
    colorlinks=true,
    linkcolor=blue,
    citecolor=blue,
    filecolor=magenta,      
    urlcolor=cyan,}

\usepackage[]{natbib}

\newtheorem{theorem}{Theorem}[section]
\newtheorem{proposition}[theorem]{Proposition}
\newtheorem{lemma}[theorem]{Lemma}

\newtheorem{corollary}[theorem]{Corollary}

\newtheorem{condition}[theorem]{Condition}
\newtheorem{remark}[theorem]{Remark}

\numberwithin{equation}{section}
\DeclareMathOperator*{\argmin}{argmin}
\DeclarePairedDelimiter\Floor\lfloor\rfloor

\title{\LARGE Weak convergence of the sequential empirical copula processes under long-range dependence}

\author[1]{Yusufu Simayi\thanks{Email: y.simayi@math.uni-goettingen.de; 
Alias: Yusup Ismayil}}

\affil[1]{\small Institute for Numerical and Applied mathematics, University of G\"{o}ttingen, 
Lotzestr. 16-18, 37083, G\"{o}ttingen Germany;\\
Email: y.simayi@math.uni-goettingen.de}

\date{\today}

\begin{document}

\maketitle \baselineskip 12pt

\begin{abstract}
We consider multivariate copula-based stationary time-series under Gaussian subordination.
Observed time series are subordinated to long-range dependent Gaussian processes and characterized by arbitrary marginal 
copula distributions. First of all, we establish limit theorems for the marginal and quantile marginal 
empirical processes of multivariate stationary long-range dependent sequences under Gaussian
subordination. Furthermore, we establish the asymptotic behavior of sequential empirical copula processes under non-restrictive
smoothness assumptions. The limiting processes in the case of long-memory sequences are quite different from the cases of 
of i.i.d. and weakly dependent observations.\\


\end{abstract}

\section {Introduction}

\hspace {0.5cm} In recent years, it is increasingly popular to investigate the estimations of multivariate
stationary time series models by copula function. The parametric copula function capture the cross-sectional dependence between
time series, which provide us an important way to construct a multivariate time series models.
By using copulas to model the temporal dependence between the multivariate
time series and the applications of copulas in modelling financial has been widely studied in many literature,
see, e.g., \cite{Joe1997} and \cite{Patton2009}. Moreover, it is a main tool in the field of finance, 
Hydrology and insurance risk management, see, e.g., \cite{McNeil2005}, \cite{Frees1998} respectively.

Let $F$ be a joint distribution function with continuous marginal dirstributions $F_{1}, \cdots, F_{p}$ and a copula $C$ associated to $F$. 
Then by the Sklar theorem in \cite{Nelsen2006}, the joint distribution $F$ decompose as:
\begin{equation*}
F(x_{1}, \cdots, x_{p})=C(F_{1}(x_{1}), \cdots, F_{p}(x_{p})), \quad\text{for all}~ \mathbf{x}=(x_{1}, \cdots, x_{p})\in \mathbb{R}^{p}.
\end{equation*}
To estimate such a copula $C$, (sequential) empirical copula $C_{\Floor{nz}}$ is a main tool for statistical inference on unknown copulas. 
Let $\mathbf{X}_{1}, \cdots,\mathbf{X}_{n}$ be $p$-dimensional random vectors from strictly stationary sequence 
$(\mathbf{X}_{t})_{t\in\mathbb{Z}}$ with copula $C$. If we denotes $F_{\Floor{nz}}$ the empirical distribution function of the random
vectors $\mathbf{X}_{t}$, then the sequential empirical copula is defined as
\begin{equation*}
C_{\Floor{nz}}(\mathbf{u})=F_{\Floor{nz}}(F_{\Floor{nz},1}^{-1}(u_{1}), \cdots, F_{\Floor{nz},p}^{-1}(u_{p})), 
\quad \mathbf{u}=(u_{1}, \cdots, u_{p})\in[0, 1]^{p},
\end{equation*}
where $F_{\Floor{nz}}$ and $F_{\Floor{nz},j}$ are joint empirical and $j$-th marginal empirical distribution functions of sample 
respectively, and $F_{\Floor{nz},j}^{-1}$ is $j$-th sequential quantile marginal empirical distributions, see Section 2.2 for more details.
In this paper, we mainly consider the following sequential empirical copula process
\begin{equation}\label{1.1}
\mathbb{C}_{\Floor{nz}}(\mathbf{u})=\frac{1}{q_{n}}\sum_{t=1}^{\Floor{nz}}\left\{\mathbb{I}\{\mathbf{X}_{t}\leq 
F_{\Floor{nz}}^{-1}(\mathbf{u})\}-C(\mathbf{u})\right\} \quad\text{for}~ \mathbf{u}\in[0, 1]^{p}, z\in[0,1],
\end{equation}
where $q_{n}$ is a normalizing factor to be discussed in Section 2.  
To establish asymptotic behavior of the empirical copula process are fundamental step to estimate parametric 
copulas and testing problems. 

As we know, there exists a number of literature for studying the
asymptotic behaviors of the empirical copula processes in the context of i.i.d sequences, i.e., \cite{Ruschendorf1976},
\cite{Gaenssler1987}, \cite{vanderVaart1996},\cite{Fermanian2004},\cite{Tsukahara2005}, \cite{Bucher2011}, 
\cite{Segers2012}. More specifically, \cite{Gaenssler1987}, \cite{Tsukahara2005} and
\cite{Segers2012} prove convergence results of $\mathbb{C}_{n}$ (where $z=1$) by directly analyzing the behavior of the empirical
copula processes $\mathbb{C}_{n}$. The rest of the authors applied a functional delta method to derive the weak convergence of
$\mathbb{C}_{n}$. They analyze the mapping properties between spaces of functions and made a smoothness assumption that the copula $C$
has continuous partial derivatives on the closd unit cube, see e.g. \cite{Segers2012}. 

Furthermore, in the case of weakly dependent stationary sequences, the sequential empirical copula processes 
$\mathbb{C}_{\Floor{nz}}$ with order $q_{n}=o(n^{1/2})$: 
\begin{equation}\label{1.2}
\mathbb{C}_{\Floor{nz}}(\mathbf{u})=\sqrt{n}(C_{\Floor{nz}}(\mathbf{u})-C(\mathbf{u}))
\end{equation}
converges weakly to a Gaussian field $\mathbb{G}_{C}$ with a summable covariance functions, see, e.g., 
\cite{Doukhan2008}, \cite{Remillard2010}, \cite{Rmillard2012}, \cite{BucherRuppert2013} and \cite{Bucher2013}.
More precisely, \cite{Doukhan2008} establish limit theorems for $\mathbb{C}_{n}$ under $\eta$-dependence, \cite{BucherRuppert2013}
derive the asymptotic behavior of $\mathbb{C}_{n}$ for $\alpha$-mixing conditions, \cite{Rmillard2012} aslo establish
limit theorems for copula-based semiparametric models for multivariate time series under both interdependence and serial 
dependence. \cite{Bucher2013} also derive asymptotic behavior of empirical and sequential empirical copula process (\ref{1.2}) 
under general setting by using the functional delta method. As for the long-range dependent sequences, 
\cite{Ibragimov2017} and \cite{Chicheportiche2014} study the copula-based time series with long-memory and provide applications, but
they didn't discuss the asymptotic properties of the empirical copula processes and establish limit theorems.
\cite{Bucher2013} sign us that the derivation of weak convergence results under long-range dependence 
can be established based on the their method.

Recently, \cite{Beran2016} construct multivariate copula-based stationary time series models under 
Gaussian subordination and estbalish the limit theorem for a nonparametric estimator of the Pickands dependence function. 
The copula-based time series with cross-sectional dependence characterized by a copula function and subordinated 
to the nonlinear long-memory Gaussian processe. To the best of our knowledge, there are no further results 
have been found for the asymptotic behavior of the (sequential) empirical 
copula processes $\mathbb{C}_{\Floor{nz}}$ in (\ref{1.1}) under long-range dependence. In the case of long-range dependent sequences, 
the limiting distribution is not a Gaussian anymore as well as the covariance function is not summable.
To establish a limit theorem for the sequential empirical copula processes
$\mathbb{C}_{\Floor{nz}}$ we first purpose to establish asymptotic properties of the sequential marginal empirical processes.
This process plays an important role in nonparametric statistics such as change-point analysis and testing issue. 

However, in the case of long-memory sequences, the limiting distributions and the asymptotic behavior are very
different from the cases of i.i.d. and short memory sequences and asymptotically not normal.  For literature on
the asymptotic properties of the empirical processes with one dimensional case, see e.g., \cite{Taqqu1975}, \cite{Dehling1989}
and \cite{Leonenko2001}. For the bivariate and multivariate cases, see e.g., \cite{Marinucci2005}, \cite{Taufer2015} and 
\cite{Arcones1994}, \cite{Bai2013}, \cite{Buchsteiner2015} and \cite{Mounirou2016}, respectively. The usual way to derive a
limiting distribution of multivariate empirical processes in the context of long-memory is using 
the multivariate unifrom reduction principle. 

The remainder of this paper is organized as follows. In section 2, we first present a multivariate copula-based stationary time series
with long-memory and Hermite polynomials expansions of the empirical processes. 
In section 3, we establish the weak convergence of the sequential marginal and quantile marginal empirical processes. 
In section 4, we derive the weak convergence of usual empirical empirical processes with different Hermite ranks.
In section 5, we derive the asymptotic behavior of the sequential empirical copula processes $\mathbb{C}_{\Floor{nz}}$.
In section 6, we address our conclusions and some open problems.

\section{Empirical copula processes under long-range dependence}

\subsection{Multivariate Copula-based long-range dependent sequences}

\hspace {0.5cm}
Let $\mathbf{X}_{t}=(X_{t,1}, \cdots, X_{t, p})_{t\in\mathbb{Z}}^{'}$ be a vector valued strictly stationary process in $\mathbb{R}^{p}$.
Let $F$ be the joint multivariate distribution function of $\mathbf{X}_{t}$ with continuous marginal distributions $F_{1}, \cdots, F_{p}$
for all $t\in\mathbb{Z}, j=1,\cdots, p$, then by Sklar theorem in \cite{Nelsen2006}, there exists an unique copula $C$ such that
\begin{equation}\label{2.1}
F(x_{1}, \cdots, x_{p})=C(F_{1}(x_{1}),\cdots, F_{p}(x_{p})), \quad\mbox{for all}~ x_{j}\in\mathbb{R}, j=1, \cdots, p,
\end{equation}
where the marginal distributions $F_{1}, \cdots, F_{p}$ are standard uniform on $[0,1]$. We now consider that the observable
stationary time series $\{X_{t,j}\}_{t\in\mathbb{Z}}$ are subordinated to the long-memory Gaussian processes.

\begin{condition}\label{Condition-2.1}
Assume that $\{\eta_{t, j}\}_{t\in\mathbb{Z}}, j=1, \cdots, p$ are unobservable, independent univariate stationary 
Gaussian processes with $\mathbf{E}[\eta_{t, j}]=0, \mathbf{Var}[\eta_{t, j}]=1$ and the autocovariance functions satisfying
\begin{align*}
\gamma_{\eta, j}(k)&:=\mbox{Cov}(\eta_{t,j}, \eta_{t+k, j})\sim L_{\eta,j}(k)|k|^{2d_{j}-1}, \quad\mbox{for all}\quad d_{j}\in(0, 1/2), 
\end{align*}
where $"\sim"$ denotes that the ratio between the right and left hand sides tends to one and the functions
$L_{\eta, j}(k)$ are slowly varying at infinity and are positive for all large $k$, i.e.,
\begin{equation*}
\lim_{y\rightarrow\infty}\frac{L_{\eta, j}(\alpha y)}{L_{\eta, j}(y)}=1, \quad \forall \alpha>0.
\end{equation*}
\end{condition}

Since the univariate Gaussian processes $\eta_{t, j}$ are mutually independent, the spectral density of $X_{t, j}$ are
defined by 
\begin{equation*}
f_{X, j}(k)\sim c_{X, j}|k|^{-2d_{j}} \quad (\text{as}~ |k|\rightarrow 0).
\end{equation*}
for some constants $0<c_{X, j}<\infty$.
Let $G(\eta_{t}):=(G_{1}(\eta_{t}), \cdots, G_{p}(\eta_{t}))^{'}: \mathbb{R}^{p}\mapsto \mathbb{R}^{p}$
be Borel measurable functions satisfying with $\mathbf{E}[G_{j}(\eta_{t,j})]=0$ and $\mathbf{E}[G_{j}^{2}(\eta_{t})]<\infty$
for $j=1, \cdots, p$. Then the subordinated processes $(\mathbf{X}_{t})_{t\in\mathbb{Z}}$ are defined by
\begin{equation*}
\mathbf{X}_{t}=G(\eta_{t})=(G_{1}(\eta_{t}), \cdots, G_{p}(\eta_{t}))^{'}
\end{equation*}
in $\mathbb{R}^{p}$, where $G(\eta_{t})$ is nonlinear functions of stationary Gaussian processes.

Note that the memory parameters $d_{j}\in (0, \frac{1}{2})$ are not necessarily all equal 
and suppose to be known. For the memory parameters $d_{j}\in(0,1/2)$, the stationary Gaussian sequences
$(\eta_{t,1}, \cdots, \eta_{t,p})$ display long-memory behavior and have nonsummable autocovariance functions, i.e.,
$\sum_{k\in\mathbb{Z}}\gamma_{\eta, j}(k)=\infty$.  If $d_{j}=0$,
the autocovariance functions $0<\sum_{k\in\mathbb{Z}}\gamma_{\eta, j}(k)<\infty$,
then the Gaussian processes $\eta_{t,j}$ exhibit short-memory behavior. For more details of linear and nonlinear long-memory 
processes and applications, we refer to \cite{Beran1994}, \cite{Beran2013}. For the simplicity, here we only consider the 
case with $d_{j}=d, j=1, \cdots, p$. 

According to Lemma 1 in \cite{Beran2016}, we find that the observed process $(\mathbf{X}_{t})_{t\in\mathbb{Z}}$ 
has a copula $C$ as in (\ref{2.1}) and the multivariate Hermite rank of each $G_{j}$ is one. More precisely, we have
\begin{equation}\label{2.2}
\mathbf{X}_{t}=\begin{pmatrix}
                X_{t,1}\\
                X_{t,2}\\
                \vdots\\
                X_{t, p}
               \end{pmatrix}:=
               \begin{pmatrix}
                G_{1}(\eta_{t})\\
                G_{2}(\eta_{t})\\
                \vdots\\
                G_{p}(\eta_{t})
               \end{pmatrix}=
                \begin{pmatrix}
                F_{1}^{-1}(\Phi(\eta_{t,1}))\\
                F_{2}^{-1}(C_{2|1}^{-1}(\Phi_{2}(\eta_{t,2})|\Phi_{1}(\eta_{t,1})))\\
                \vdots\\
                F_{p}^{-1}(C_{p|1, \cdots, p-1}^{-1}(\Phi_{p}(\eta_{t,p})|\Phi_{1}(\eta_{t, 1}), \cdots, \Phi_{p-1}(\eta_{t, p-1})))
               \end{pmatrix},
\end{equation}
where $\Phi(\cdot)$ denotes the cumulative standard normal distribution and $C_{j|1, \cdots, j-1}^{-1}$ the conditional quantile copula
of $U_{j}=F_{j}(X_{j})$ given $U_{1}, \cdots, U_{j-1}$, for the deatils, see \cite{Beran2016}.  

It can be seen from (\ref{2.2}) that the functions $G_{j}$ are monotonically nondecreasing.
Let $L^{2}$ be the space of square integrable functions with respect to $p$-dimensional standard normal distribution $\eta_{t}
=(\eta_{t, 1}, \cdots, \eta_{t, p})^{'}$. An orthogonal basis of these space is given by 
\begin{equation*}
H_{r_{1},\cdots, r_{p}}(\eta_{t})=H_{r_{1}}(\eta_{t, 1})\cdots H_{r_{p}}(\eta_{t, p}), \quad r_{1}, \cdots, r_{p}\in \mathbb{N},
\end{equation*}
with Hermite rank $r=r_{1}+\cdots+r_{p}$, where $H_{r_{j}}$ is an one-dimensional Hermite polynomial of order $r_{j}$, i.e., 
\begin{equation*}
H_{r_{j}}(x)=(-1)^{r_{j}}\exp\left(\frac{x^{2}}{2}\right)\frac{\mathrm{d}^{r_{j}}}{\mathrm{d}x^{r_{j}}}\exp\left(-\frac{x^{2}}{2}\right),
\quad r_{j}=0,1 ,2,\cdots, x\in\mathbb{R}.
\end{equation*}
with $H_{0}(x)=1, H_{1}(x)=x, H_{2}(x)=x^{2}-1$ etc. Thus, 
the Hermite coefficient of $G_{j}$ with respect to the Hermite Polynomials $H_{e(l)}(\eta)=\eta_{l}$ is defined by
\begin{equation}\label{2.3}
J_{j, l}(H, G)=\langle H_{e(l)}(\eta_{t}), G_{j}(\eta_{t})\rangle=\mathbf{E}[\eta_{t, l}G_{j}(\eta_{t})].
\end{equation}
where $e(l)=(e_{1}(l), \cdots, e_{p}(l)^{'}$ in $\mathbb{N}^{p}$ is $l$th unit vector.

\subsection{Hermite polynomials expansions of the empirical copula processes}
\hspace{0.5cm}
Let $(\mathbf{X}_{1},  \cdots, \mathbf{X}_{n})$ be samples from the Gaussian subordinated long-memory processes 
$\mathbf{X}_{t}$. Define $U_{t, j}=F_{j}(X_{t, j})$ with $X_{t, j}=G_{j}(\eta_{t})$ for $t\in \{1, \cdots, n\}$
 and $j=1,\cdots, p$. Let $\mathbf{U}_{t}=(U_{t, 1}, \cdots, U_{t, p})^{'}$ be the random vectors 
sampled from copula $C$. The corresponding sequential empirical distribution functions are given by
respectively. 
\begin{align*}
F_{\Floor{nz}, j}(x_{j})=
\begin{cases}
0, &z\in[0, \frac{1}{n}),\\               
\frac{1}{\Floor{nz}}\sum\limits_{t=1}^{\Floor{nz}}\mathbb{I}
\{G_{j}(\eta_{t})\leq x_{j}\}, &x_{j}\in\mathbb{R},~z\in[\frac{1}{n}, 1].             
\end{cases}
\end{align*}
and 
\begin{equation*}
D_{\Floor{nz}, j}(u_{j})=
\begin{cases}
0, &z\in[0, \frac{1}{n}),\\               
\frac{1}{\Floor{nz}}\sum\limits_{t=1}^{\Floor{nz}}\mathbb{I}
\{F_{j}G_{j}(\eta_{t})\leq u_{j}\}, &u_{j}\in[0,1],~z\in[\frac{1}{n}, 1].             
\end{cases}
\end{equation*}
for $\mathbf{x}=(x_{1}, \cdots, x_{p})^{'}\in\mathbb{R}^{p}$ and $\mathbf{u}=(u_{1}, \cdots, u_{p})^{'}\in[0, 1]^{p}$ and $z\in[0,1]$.
The marginal quantile empirical distributions associated to $F_{\Floor{nz},j}$ and $D_{\Floor{nz},j}$ are
\begin{equation*}
F_{\Floor{nz},j}^{-1}(u_{j})=\inf\{x\in\mathbb{R}: F_{\Floor{nz},j}(x_{j})\geq u_{j}\}=
\begin{cases}
\widehat{X}_{t,j}, &\text{if}~ u_{j}\in (0, 1], ~z\in[0, 1],\\
-\infty,         &\text{if}~ u_{j}=0, z\in[0, 1].
\end{cases}
\end{equation*}
and
\begin{align*}
D_{\Floor{nz},j}^{-1}(u_{j})=\inf\{u_{j}\in[0,1]: D_{\Floor{nz},j}(u_{j})\geq u_{j}\}&=
\begin{cases}
\widehat{U}_{t,j}, &\mbox{if}~ u_{j}\in(0, 1],~ z\in[0, 1],\\
0, &\mbox{if}~ u_{j}=0, ~z\in[0, 1],
\end{cases}
\end{align*}
where $\widehat{X}_{t,j}$ and $\widehat{U}_{t,j}$ are the vectors of ascending order statistics of the $j$th coordinate samples
$X_{t, 1}, \cdots, X_{t, p}$ and $U_{t, 1}, \cdots, U_{t, p}$, respectively.
Let $F_{\Floor{nz}}$ be the joint empirical distribution function of $(\mathbf{X}_{1}, \cdots,\mathbf{X}_{n})$. 
Then the associated empirical copula is defined by
\begin{equation}\label{2.4}
C_{\Floor{nz}}(\mathbf{u})=F_{\Floor{nz}}(F_{\Floor{nz},1}^{-1}(u_{1}), \cdots, F_{\Floor{nz},p}^{-1}(u_{p})), 
\quad\text{for all}~ \mathbf{u}=(u_{1}, \cdots, u_{p})^{'}\in[0,1]^{p}.
\end{equation}
where
\begin{equation*}
F_{\Floor{nz}}(\mathbf{x})=\begin{cases}
0, &z\in[0, \frac{1}{n}), \\                        
\frac{1}{\Floor{nz}}\sum\limits_{t=1}^{\Floor{nz}}\mathbb{I}\{G_{1}(\eta_{t})\leq x_{1}, 
\cdots, G_{p}(\eta_{t})\leq x_{p}\},  &x_{j}\in\mathbb{R}, ~z\in[\frac{1}{n}, 1].                       
\end{cases}
\end{equation*}
for all $\mathbf{x}=(x_{1}, \cdots, x_{p})^{'}\in\mathbb{R}^{p}$.
Since $X_{t,j}\leq F_{\Floor{nz},j}^{-1}(u_{j})$ if and only if $U_{t,j}\leq D_{\Floor{nz},j}^{-1}(u_{j})$ for all $u_{j}\in[0,1], 
t=1, \cdots, n$, and $j=1, \cdots, p$, it allows us to write the empirical copula (\ref{2.4}) as
\begin{equation*}
C_{\Floor{nz}}(\mathbf{u})=D_{\Floor{nz}}(D_{\Floor{nz},1}^{-1}(u_{1}), \cdots, D_{\Floor{nz},p}^{-1}(u_{p})).
\end{equation*}
Finally, the corresponding sequential uniform and genegral processes to the empirical distribution 
functions $D_{\Floor{nz}}$ and $D_{\Floor{nz},j}($ are given as
\begin{equation}\label{2.5}
\mathbb{B}_{\Floor{nz}}(\mathbf{u})=\frac{\Floor{nz}}{q_{n}}\left\{D_{\Floor{nz}}(\mathbf{u})-C(\mathbf{u})\right\}, 
\quad \mathbb{B}_{\Floor{nz},j}(u_{j})=\frac{\Floor{nz}}{q_{nj}}\left\{D_{\Floor{nz},j}(u_{j})-u_{j}\right\},
\end{equation}
and the processes to the empirical quantile $D_{\Floor{nz},j}^{-1}$ and empirical copula distribution $C_{\Floor{nz}}$
are give as
\begin{equation}\label{2.6}
\mathbb{C}_{\Floor{nz}}(\mathbf{u})=\frac{\Floor{nz}}{q_{n}}\left\{C_{\Floor{nz}}(\mathbf{u})-C(\mathbf{u})\right\}, 
\quad \mathbb{Q}_{\Floor{nz},j}(u_{j})=\frac{\Floor{nz}}{q_{nj}}\left\{D_{\Floor{nz},j}^{-1}(u_{j})-u_{j}\right\},
\end{equation}
for all $u_{j}\in[0, 1], z\in[0, 1]$, where $q_{n}$ and $q_{nj}$ are the normalizing factors 
corresponding to the joint and marginal empirical distributions, respectively, which will be given later.

In the following, we need some notations as given in \cite{Dehling1989} for expanding the class of functions
$\mathbb{I}\{\mathbf{U}_{t}\leq \mathbf{u}\}-C(\mathbf{u})$ and $\mathbb{I}\{U_{t,j}\leq u_{j}\}-u_{j}$, respectively.
Let $H_{r}(\cdot)$ be denote the $r$-th order Hermite polynomials with $H_{r}(x)=H_{r_{1}}(x_{1})\cdots H_{r_{p}}(x_{p})$.
Since the functions $H_{r_{1}}(x_{1})\cdots H_{r_{p}}(x_{2})$ form a orthogonal system in $L^{2}$-space with
\begin{equation*}
L^{2}(\mathbb{R}^{p}, \phi(x_{1})\cdots, \phi(x_{p}))=\left\{G: \mathbb{R}^{p}\mapsto\mathbb{R},~\text{and}~\|G\|^{2}<\infty\right\},
\end{equation*}
where $\|G\|^{2}=\int_{\mathbb{R}^{2}} G^{p}(\mathbf{x})\phi(x_{1})\cdots\phi(x_{p})\mathrm{d}x_{1}\cdots
\mathrm{d}x_{p}<\infty$ and $\phi(\cdot)$ is standard Gaussian density function. From the orthogonality properties of Hermite
polynomials for a zero-mean, unit variance of Gaussian variables $\eta_{t, 1}, \cdots, \eta_{t,p}$, we have
\begin{align}\label{2.7}
\mathbf{E}[H_{rt}(\eta_{t,k})H_{rt}(\eta_{t,l})]=r_{l}!\delta_{rl}^{rk}\{E[\eta_{t,l}\eta_{t,k}]\}^{rt}, \quad
\delta_{rl}^{rk}=\begin{cases}
1 &\mbox{for}~r_{l}=r_{k}\\
0 &\mbox{for}~r_{l}\neq r_{k}
\end{cases},
\end{align}
where $\delta_{rl}^{rk}$ is Kronecker's delta. By Condition \ref{Condition-2.1}, we have
\begin{align*}
\sigma_{r_{1}\cdots r_{p}}^{2}(n)&=\frac{1}{r_{1}!\cdots r_{p}!}\mathbf{E}\left[\frac{1}{n}\sum_{t=1}^{n}H_{r_{1}}
(\eta_{t,1})\cdots H_{r_{p}}(\eta_{t, p})\right]^{2}\nonumber\\
&=\frac{1}{n^{2}}\sum_{t_{1}=1}^{n}\cdots\sum_{t_{p}=1}^{n}\gamma_{\eta, {t_{1}}}^{r_{1}}(|t_{1}-s_{1}|)\cdots 
\gamma_{\eta,t_{p}}^{r_{p}}(|t_{p}-s_{p}|). 
\end{align*}
where
\begin{align*}
\sigma_{r_{1}\cdots r_{p}}^{2}(n)&
\sim \begin{cases}
a(r_{1} \cdots r_{p}, d_{1} \cdots d_{p}) L_{\eta, 1}^{r_{1}}(n)\cdots L_{\eta, p}^{r_{p}}(n)n^{p-\sum\limits_{j=1}^{p}r_{j}(2d_{j}-1)},
&\sum\limits_{j=1}^{p}r_{j}(1-2d_{j})<1,\\
an^{-1}, &\sum\limits_{j=1}^{p}r_{j}(1-2d_{j})>1,
     \end{cases}
\end{align*}
as $n\rightarrow \infty$. According to the argument as in Taqqu in Theorem 3.1 \cite{Taqqu1975} or \cite{Mounirou2016} and view of
(\ref{2.1}) and (\ref{2.7}), we have
\begin{equation*}
a(r_{1} \cdots r_{p}, d_{1} \cdots d_{p})=\frac{pr_{1}!\cdots ! r_{p}!}
{\prod_{j=1}^{p}[j-r_{1}(1-2d_{1})-\cdots-r_{p}(1-2d_{p})]}.
\end{equation*}
If the memory parameter are all equal $d=d_{1}=\cdots=d_{p}$ and $r_{1}+\cdots+r_{p}=r$, then 
\begin{equation*}
a(r_{1}\cdots r_{p}; d)=\frac{pr_{1}!\cdots r_{p}!}{[1-r(1-2d)]\cdots[p-r(1-2d)]}.
\end{equation*}

Now we expand the class of joint distribution function $\mathbb{I}\{\mathbf{U}_{t}\leq \mathbf{u}\}-C(\mathbf{u})$
in Hermite polynomials for any fixed $\mathbf{u}\in[0, 1]^{p}$:
\begin{equation}\label{2.8}
\mathbb{I}\{\mathbf{U}_{t}\leq \mathbf{u}\}-C(\mathbf{u})=\sum_{r=1}^{\infty}\cdots\sum_{r_{1}+\cdots+r_{p}=r}^{\infty}
\frac{J_{r_{1}\cdots r_{p}}(\mathbf{u})}{r_{1}!\cdots ! r_{p}!}\prod_{j=1}^{p}H_{r_{j}}(\eta_{t,j}), 
\end{equation}
with the Hermite coefficient $J_{r_{1}\cdots r_{p}}(\mathbf{u})$, i.e.,
\begin{equation}\label{2.9}
J_{r_{1}\cdots r_{p}}(\mathbf{u})=\mathbf{E}[\mathbb{I}(\mathbf{U}_{t}\leq \mathbf{u})]\prod_{j=1}^{p}H_{r_{j}}(\eta_{t,j}),
\end{equation}
where $r$ is the Hermite rank of the function (\ref{2.9}), i.e.,
$r=r(\mathbf{u})=\min\{r_{1}+\cdots+r_{p}=r: J_{r_{1}\cdots r_{p}}(\mathbf{u})\neq 0~ \mbox{for all}~ u_{j}\in [0,1]\}$. As described in
\cite{Marinucci2005}, we find that the stochastic order of magnitude of $C_{n}(\mathbf{u})$ is determined by the lowest
 $r_{1}(1-2d_{1})+\cdots+r_{p}(1-2d_{p})$ terms corresponding to $J_{r_{1}, \cdots, r_{p}}(\mathbf{u})$. We define
\begin{equation*}
(r_{i1}^{*}, \cdots, r_{ip}^{*})=\argmin\{r_{1}(1-2d_{1})+\cdots+r_{p}(1-2d_{p})\}
\end{equation*}
such that the Hermite coefficient $J_{r_{i1}^{*} \cdots r_{ip}^{*}}(\mathbf{u})\neq 0$ for $i=1, \cdots, M$. If all the memory paramters
are equal, i.e., $d_{j}=d$, then the stochastic order $r(1-2d)$ is constant for all $i$ and the Hermite rank of 
$\mathbb{I}\{\mathbf{U}_{t}\leq \mathbf{u}\}-C(\mathbf{u})$ follows from the result of \cite{Arcones1994}. 

\begin{condition}\label{Condition-4.1}
(regularity assumption). Let 
\begin{equation*}
q_{n}(r_{i1}^{*}, \cdots, r_{ip}^{*})=a(r_{i1}^{*}, \cdots, r_{ip}^{*}; d_{1}, \cdots, d_{p})^{\frac{1}{2}}
L_{\eta, 1}^{\frac{r_{i1}^{*}}{2}}(n)\cdots L_{\eta, p}^{\frac{r_{ip}^{*}}{2}}(n)n^{1-\frac{r_{i1}^{*}+\cdots+r_{ip}^{*}}{2}}.
\end{equation*}
be the square root of the asymptotic variance of $\sum_{t=1}^{n}\prod_{j=1}^{p}H_{j}(\eta_{t,j})$
 with $r_{i1}^{*}(1-2d_{1})+\cdots+r_{ip}^{*}(1-2d_{p})<1$.
Suppose that the exist some finite and positive constants  $\rho_{1}, \cdots, \rho_{M}>0$. Then 
\begin{equation*}
\lim_{n\rightarrow \infty}\frac{q_{n}(r_{i1}^{*}, \cdots, r_{ip}^{*})}{q_{n}(r_{11}^{*}, \cdots, r_{1p}^{*})}=:\rho_{i}, \quad
i=1, 2, \cdots, M
\end{equation*}
exists and non-zero. Moreover, we have
\begin{equation*}
\rho_{i}=\frac{a(r_{i1}^{*}, \cdots, r_{ip}^{*}; d_{1}, \cdots, d_{p})}{a(r_{11}^{*}, \cdots, r_{1p}^{*}; d_{1}, \cdots, d_{p})}
\lim_{n\rightarrow\infty}\frac{L_{\eta, 1}^{r_{i1}^{*}/2}(n)\cdots L_{\eta, p}^{r_{ip}^{*}/2}(n)}
{L_{\eta, 1}^{r_{11}^{*}/2}(n)\cdots L_{\eta, p}^{r_{1p}^{*}/2}(n)}. 
\end{equation*}
\end{condition}

Define the random processes, which can be represented by multiple Wiener-It\^{o} integrals, see e.g., \cite{Dobrushin1979} or 
\cite{Mounirou2016}:
\begin{align}\label{2.10}
\mathbb{H}_{r_{1}\cdots r_{p}}(z)&=\int_{R^{r_{1}}}^{'}\cdots\int_{R^{r_{p}}}^{'}
h(z; \lambda_{1}^{2},\cdots, \lambda_{r_{2}}^{2})\cdots h(z; \lambda_{1}^{p},\cdots, \lambda_{r_{p-1}}^{p})\nonumber\\
&\prod_{j_{1}=1}^{r_{1}}|\lambda_{j_{1}}|^{(d_{1}-1)}\cdots\prod_{j_{p}=1}^{r_{p}}|\lambda_{j_{p}}|^{(d_{p}-1)}
\prod_{j_{1}=1}^{r_{1}}W_{1}(\mathrm{d}\lambda_{j_{1}})\cdots\prod_{j_{p}=1}^{r_{p}}W_{p}(\mathrm{d}\lambda_{j_{p}})
\end{align}
and
\begin{equation*}
h(z; \lambda_{1}^{j},\cdots, \lambda_{r_{j}}^{j})=\frac{1}{\sqrt{a(r_{1}, \cdots, r_{p} ,d_{1}, \cdots, d_{p})}}
\frac{\exp(iz(\lambda_{1}^{j}+\cdots+\lambda_{r_{j}}^{j}))-1}{i(\lambda_{1}^{j}+\cdots+\lambda_{r_{j}}^{j})}
\end{equation*}
where $W_{j}(\cdot),~ j=1, \cdots, p$ are independent copies of a complex-valued Gaussian white noise on $R^{1}$ and
the symbol $\int_{\mathbb{R}^{r_{j}}}^{'}$ refer that the domain of integration exclude the hyperdiagonals with
$\lambda_{j}=\pm\lambda_{j^{(1)}}, j\neq j^{(1)}$ for $j, j^{(1)}=1, \cdots, r_{j}$, $\lambda$ denote Lebesgue measure. Note that
the process $\mathbb{H}_{r,s}(z)$ defined in (\ref{2.10}) is called an $r$ th order Hermite process.
The coefficients are
\begin{equation*}
a(r_{1}\cdots r_{p}; d_{1}\cdots d_{p})=\frac{[1-r(2d-1)/2][1-r(2d-1)]}{r!\{2\Gamma(2d-1)\sin[(3/2-d)\pi]\}^{r}}.
\end{equation*}
The Hermite process $\mathbb{H}_{r_{1}\cdots r_{p}}(z)$ is called fractional Brownian motion if $r=1$. 
If $r=2$, the process $\mathbb{H}_{2}(z)$ is called Rosenblatt process, see \cite{Taqqu1975}. 
Otherwise the limiting process $\mathbb{H}_{r_{1}, \cdots r_{p}}(z)$ is Hermite process of order $r=\sum_{j=1}^{p}r_{j}$.

\section{Weak convergence of the marginal and quantile marginal empirical processes}

\hspace{0.4cm}
Before we estbalish the weak convergence of the sequential empirical copula processes, it is a main step to derive the weak convergence 
of the marginal and qunatile marginal sequential empirical processes $\mathbb{B}_{\Floor{nz},j}$ and $Q_{\Floor{nz}, j}$ with the case that
the memory parameters $d_{j}=d$ are all equal for all $j=1, \cdots, p$.

Since the marginal distributions $F_{1}, \cdots F_{p}$ are conditionally subordinated to the long-memory Gaussian processes, 
the class of functions $\mathbb{I}\{G_{j}(\eta_{t})\leq x_{j}\}-F_{j}(x_{j})$ are expanded in multivariate Hermite polynomials for fixed
$x_{j}\in\mathbb{R}$:
\begin{align*}
\mathbb{I}\{G_{j}(\eta_{t})\leq x_{j}\}-F_{j}(x_{j})&=\sum_{m_{j}=1}\sum_{\sum_{j=1}^{p}r_{j}=m_{j}}^{\infty}
\frac{T_{j(r_{1}, \cdots, r_{p})}(x_{j})}{\prod_{j=1}^{p}r_{j}!}\prod_{j=1}^{p}H_{r_{j}}(\eta_{t, j})
\end{align*}
where $j(r_{1}, \cdots, r_{p})$ is the number of indeces $r_{1}, \cdots, r_{p}$ that are equal to $j$ 
and the corresponding Hermite coefficients of these functions are denoted by
\begin{equation*}
T_{j(r_{1}, \cdots, r_{p})}(x_{j})=\mathbf{E}\left[\mathbb{I}\{G_{j}(\eta_{t})\leq x_{j}\}-F_{j}(x_{j})\right]
\prod_{j=1}^{p}H_{r_{j}}(\eta_{t, j}).
\end{equation*}
Then we can define the Hermite rank of the functions $\mathbb{I}\{G_{j}(\eta_{t})\leq x_{j}\}-F_{j}(x_{j})$ as 
\begin{equation*}
m_{j}=\min\{m_{j}(x_{j})=\sum_{j=1}^{p}r_{j}: T_{j(r_{1}, \cdots, r_{p})}(x_{j})\neq 0~ \mbox{for all}~ x_{j}\in\mathbb{R}\}.
\end{equation*}
Since the uniform marginals $F_{j}$ are continuous, the class of functions $\mathbb{I}\{F_{j}(G_{j}(\eta_{t})\leq u_{j}\}
-u_{j}$ can be expanded in Hermite polynomials for any fixed $u_{j}\in(0, 1)$:
\begin{align*}
\mathbb{I}\{F_{j}(G_{j}(\eta_{t}))\leq u_{j}\}-u_{j}&=\sum_{m_{j}=1}\sum_{\sum_{j=1}^{p}r_{j}=m_{j}}^{\infty}\frac{J_{j(r_{1}, \cdots, r_{p})}
(u_{j})}{\prod_{j=1}^{p}r_{j}!}\prod_{j=1}^{p}H_{r_{j}}(\eta_{t, j}),
\end{align*}
the corresponding Hermite coefficients are denoted by
\begin{equation*}
J_{j(r_{1}, \cdots, r_{p})}(u_{j})=\mathbf{E}\left[\mathbb{I}\{F_{j}(G_{j}(\eta_{t}))\leq u_{j}\}-u_{j}\right]
\prod_{j=1}^{p}H_{r_{j}}(\eta_{t, j}).
\end{equation*}
Since $T_{j(r_{1}, \cdots, r_{p})}(D^{-1}(u_{j}))=J_{j(r_{1}, \cdots, r_{p})}(u_{j})$ for any $u_{j}\in [0, 1]$, where 
$D^{-1}(u_{j})=\inf\{x_{j}: F_{j}(x_{j})=u_{j}\}$ for $u\in(0, 1]$. The Hermite rank of the class of functions 
$\mathbb{I}\{F_{j}(G_{j}(\eta_{t}))\leq u_{j}\}-u_{j}$ are also same the rank of $\mathbb{I}\{G_{j}(\eta_{t})\leq x_{j}\}-F_{j}(x_{j})$.

Similarly, by Theorem 3.1 of \cite{Taqqu1975}, the normalizing factors corresponding to the each marginal 
empirical processes are given by
\begin{equation}\label{3.1}
q_{nj}^{2}\sim c_{j}(r, d)\prod_{j=1}^{p}L^{r_{j}}(n)n^{2-\sum_{j=1}^{n}r_{j}(1-2d)}, 
\end{equation}
with the constants 
\begin{equation}\label{3.2}
c_{j}(r, d)=\frac{2}{\prod_{j=1}^{p}[j-r_{j}(1-2d)]}. 
\end{equation}

\cite{Dehling1989} studied the asymptotic properties of the one-dimensional empirical process, \cite{Taufer2015},
\cite{Marinucci2005} and \cite{Buchsteiner2015} also performed the bivariate and multivariate expansion of one dimensional empirical
processes in Hermite polynomials. We now present the limiting distribution of all marginal empirical processes in the following
lemma since it is useful results to establish the weak convergence of the empirical copula processes under long-range dependent data.

\begin{lemma}\label{Lemma-3.1}
Under Condition \ref{Condition-2.1} and (\ref{2.2}), the class of functions $\mathbb{I}\{X_{t,j}\leq x_{j}\}-F_{j}(x_{j})$ 
have Hermite ranks $m_{j}:=\min\{m_{j}(x_{j}): T_{j(r_{1}, \cdots, r_{p})}(x_{j})\neq 0~\text{for some}~ x_{j}\in\mathbb{R}\}$
with $0<\sum_{j=1}^{p}r_{j}(1-2d)<1$. Then
\begin{equation}\label{3.3}
\mathbb{F}_{\Floor{nz}, j}(x_{j})=\frac{\Floor{nz}}{q_{nj}}\left\{F_{\Floor{nz}}(x_{j})-F_{j}(x_{j})\right\}\rightsquigarrow 
\sum_{\sum_{j=1}^{p}r_{j}=m_{j}}\frac{T_{j(r_{1}, \cdots, r_{p})}(x_{j})}{\prod_{j=1}^{p}r_{j}!}\mathbb{Z}_{j(r_{1}, \cdots, r_{p})}(z)
\end{equation}
converges weakly in $D([-\infty, \infty])\times [0,1]$ equipped with sup-norm, where the symbol ''$\rightsquigarrow$'' 
denotes weak convergence and $q_{nj}$ is as in (\ref{3.1}). The processes 
$\mathbb{Z}_{j(r_{1}, \cdots, r_{p})}(z)$, $\sum_{j=1}^{p}r_{j}=m_{j}$ are give as multiple Wiener-It$\hat{o}$ integrals of the form
\begin{align}\label{3.4}
\mathbb{Z}_{j(r_{1}, \cdots, r_{p})}(z)&=\frac{1}{\sqrt{c_{j}(r, d)}}\int_{\mathbb{R}^{m_{j}}}^{'}\frac{e^{iz(\lambda_{1}
+\cdots+\lambda_{m_{j}})-1}}{i(\lambda_{1}+\cdots+\lambda_{m_{j}})}\prod_{j=1}^{m_{j}}|\lambda_{j}|^{d-1}\nonumber\\
&\quad\times\prod_{j_{1}=1}^{r_{1}}W_{1}(\mathrm{d}\lambda_{i})\cdots\prod_{j_{p}=r_{p-1}+1}^{r_{p}}W_{p}(\mathrm{d}\lambda_{i}).
\end{align}
where $W_{1}, \cdots, W_{p}$ are independent copies of a complex valued gaussian white noise on $\mathbb{R}$ and the coefficient
$c_{j}$ is given in (\ref{3.2}).
\end{lemma}

\begin{proof} 
(sketch). The proof is strongly associated to the weak convergence results in \cite{Dehling1989}, here we 
extend the univariate case to the multivariate cases. 
The class of functions $F_{\Floor{nz},j}(x_{j})-F_{j}(x_{j})$ can be expressed as a multivariate expansions in Hermite polynomials.

The first class of function $\mathbb{I}\{G_{1}(\eta_{t, 1})\leq x_{1}\}-F_{1}(x_{1})$ can be expanded as one dimensional
Hermite polynomials in $L^{2}$: 
\begin{align*}
\mathbb{I}\{G_{1}(\eta_{t, 1})\leq x_{1}\}-F_{1}(x_{1})&=\sum_{r_{1}=m_{1}}^{\infty}\frac{T_{r_{1}}(x_{1})}{r_{1}!}
H_{r_{1}}(\eta_{t, 1})
\end{align*}
with the Hermite rank $m_{1}=\min\{r_{1}(x_{1}): T_{r_{1}}(x_{1})\neq 0\}$, where $T_{r_{1}}(x_{1})=\mathbf{E}[\mathbb{I}
\{G_{1}(\eta_{t, 1})\leq x_{1}\}H_{r_{1}}(\eta_{t,1})$. Moreover, the results of \cite{Dehling1989} show that
\begin{equation}\label{3.5}
\frac{1}{q_{n1}}\sum_{t=1}^{\Floor{nz}}H_{r_{1}}(\eta_{t, 1})\rightsquigarrow \mathbb{Z}_{r_{1}}(z)=
\frac{1}{\sqrt{c_{1}(r, d)}}\int_{\mathbb{R}^{r_{1}}}^{'}\frac{e^{iz(\lambda_{1}+\cdots+\lambda_{r_{1}})}}
{i(\lambda_{1}+\cdots+\lambda_{r_{1}})}\prod_{j_{1}=1}^{r_{1}}|\lambda_{j_{1}}|^{d-1}\prod_{j_{1}=1}^{r_{1}}W_{1}
(\mathrm{d}\lambda_{j_{1}})
\end{equation}
in the space $D([0,1])$ equipped with supremum norm.

To the second marginal empirical process, the class of function $\mathbb{I}\{G_{2}(\eta_{t,1}, \eta_{t, 2})\leq x_{2}\}-F_{2}(x_{2})$
is square integrable with respect to the standard Gaussian density such that
we can expand the class of function $\mathbb{I}\{G_{2}(\eta_{t,1}, \eta_{t, 2})\leq x_{2}\}-F_{2}(x_{2})$ with bivariate Hermite
polynomials in $L^{2}$:
\begin{align*}
\mathbb{I}\{G_{2}(\eta_{t,1}, \eta_{t, 2})\leq x_{2}\}-F_{2}(x_{2})&=\sum_{m_{2}=1}^{\infty}
\sum_{r_{1}+r_{2}=m_{2}}^{\infty}\frac{T_{r_{1}, r_{2}}(x_{2})}{r_{1}!r_{2}!}H_{r_{1}}(\eta_{t, 1})H_{r_{2}}(\eta_{t, 2})
\end{align*}
with the Hermite rank $m_{2}=m_{2}(x_{2})=\min\{r_{1}+r_{2}=m_{2}(x_{2}): T_{r_{1}, r_{2}}(x_{2})\neq 0\}$, where 
$T_{r_{1}, r_{2}}(x_{2})=\mathbf{E}[\mathbb{I}\{G_{1}(\eta_{t, 1}, \eta_{t, 2})\leq x_{2}\}H_{r_{1}}(\eta_{t,1})H_{r_{2}}(\eta_{t, 2})$. 
Similarly, by the results of \cite{Dehling1989}, we have
\begin{equation*}
\frac{1}{q_{n2}}\sum_{t=1}^{\Floor{nz}}H_{r_{1}}(\eta_{t, 1})H_{r_{2}}(\eta_{t, 2})\rightsquigarrow \mathbb{Z}_{r_{1}, r_{2}}(z),
\end{equation*}
in the space $D([0,1])$ equipped with supremum norm, where 
\begin{equation}\label{3.6}
\mathbb{Z}_{r_{1},r_{2}}(z)=\frac{1}{\sqrt{c_{2}(r, d)}}\int_{\mathbb{R}^{m_{2}}}^{'}\frac{e^{iz(\lambda_{1}+\cdots+\lambda_{m_{2}})}}
{i(\lambda_{1}+\cdots+\lambda_{m_{2}})}\prod_{j_{1}=1}^{m_{2}}|\lambda_{j_{1}}|^{d-1}\prod_{j_{1}=1}^{r_{1}}W_{1}(\mathrm{d}\lambda_{j_{1}})
\prod_{j_{2}=1}^{r_{2}}W_{2}(\mathrm{d}\lambda_{j_{2}}).
\end{equation}
see also Proposition 2 in \cite{Taufer2015}.

Consequently, for the $p$th marginal empirical process, the function $\mathbb{I}\{G_{p}(\eta_{t,1},\cdots,\eta_{t, p})\leq x_{p}\}
-F_{2}(x_{p})$ is $p$th integrable w.r.t. the standard Gaussian densities $(\phi(x_{1}, \cdots, \phi(x_{p})$
and let $L^{2}=L^{2}(\mathbb{R}, \phi(x_{1})\cdots\phi(x_{p})
\mathrm{d}x_{1}\cdots\mathrm{d}x_{p})$ be the Hilbert space of real measurable functions $G^{2}(\mathbf{x})$. 
Then there exists an expansion in $L^{2}$ for any $x_{p}\in\mathbb{R}$: 
\begin{align*}
\mathbb{I}\{G_{p}(\eta_{t,1},\cdots,\eta_{t, p})\leq x_{p}\}-F_{p}(x_{p})&=\sum_{m_{p}=1}^{\infty}\sum_{r_{1}+\cdots+r_{p}=
m_{p}}^{\infty}\frac{T_{r_{1},\cdots, r_{p}}(x_{p})}
{r_{1}!\cdots r_{p}!}H_{r_{1}}(\eta_{t, 1})\cdots H_{r_{p}}(\eta_{t, p})
\end{align*}
with the Hermite rank $m_{p}=m_{p}(x_{p})=\min\{r_{1}+\cdots+r_{p}=m_{p}: T_{m_{p}}(x_{p})\neq 0\}$, where 
$T_{r_{1}, \cdots, r_{p}}(x_{p})=\mathbf{E}[\mathbb{I}\{G_{p}(\eta_{t, 1},\cdots \eta_{t, p})\leq x_{p}\}H_{r_{1}}(\eta_{t,1})\cdots
H_{r_{p}}(\eta_{t, p})$. Similarly, as a consequence of theorem 1 in \cite{Dehling1989}, we have
\begin{equation*}
\frac{1}{q_{np}}\sum_{t=1}^{\Floor{nz}}H_{r_{1}}(\eta_{t, 1})\cdots H_{r_{p}}(\eta_{t, p})\rightsquigarrow 
\mathbb{Z}_{r_{1}, \cdots, r_{p}}(z),
\end{equation*}
in the space $D([0,1])$ equipped with supremum norm, where
\begin{align}\label{3.7}
\mathbb{Z}_{r_{1},\cdots, r_{p}}(z)&=\frac{1}{\sqrt{c_{p}(r, d)}}\int_{\mathbb{R}^{m_{p}}}^{'}\frac{e^{iz(\lambda_{1}+\cdots+\lambda_{m_{p}})}}
{i(\lambda_{1}+\cdots+\lambda_{m_{p}})}\prod_{j_{1}=1}^{m_{p}}|\lambda_{j_{1}}|^{d-1}\nonumber\\
&\quad\times\prod_{j_{1}=1}^{r_{1}}W_{1}(\mathrm{d}\lambda_{j_{1}})
\cdots\prod_{j_{p}=1}^{r_{p}}W_{2}(\mathrm{d}\lambda_{j_{p}}).
\end{align}
Therefore, combining (\ref{3.5}), (\ref{3.6}) and (\ref{3.7}), we can conclude that the sequential marginal empirical processes (\ref{3.3})
weakly converges to the Hermite process as in (\ref{3.4}).
\end{proof}

\begin{remark}
The Hermite processes $\mathbb{Z}_{j(r_{1}, \cdots, r_{p})}(1)$ with Hermite ranks $m_{j}=\sum_{j=1}^{p}r_{j}$ are Gaussian for
$\sum_{j=1}^{p}r_{j}=1$ and the normalizing factors $c_{j}(r,d)$ ensures unit variance of $\mathbb{Z}_{j(r_{1}, 
\cdots, r_{p})}(1)$ and otherwise $\mathbb{Z}_{j(r_{1}, \cdots, r_{p})}(1)$ are non-Gaussian for $m\geq 2$.
\end{remark}

As a consequence of Lemma \ref{Lemma-3.1}, we can now immidiately establish the weak convergence of the each uniform marginal empirical 
processes $\mathbb{B}_{\Floor{nz}, j}(u_{j})$. Since the uniform marginal distributions $F_{1}, \cdots, F_{p}$ are all continuous, 
we have 
\begin{equation*}
\mathbb{B}_{\Floor{nz},j}(u_{j})=\mathbb{F}_{\Floor{nz}, j}(F_{j}^{-1}(u_{j})), \quad \text{for all}~ u_{j},z\in[0, 1].
\end{equation*}
This means that the weak convergence results that we show in Lemma \ref{Lemma-3.1} also holds true simultaneously for all
marginal empirical processes $\mathbb{F}_{\Floor{nz},j}(x_{j})$ and $\mathbb{B}_{\Floor{nz},j}(u_{j})$ if $F_{j}$ are continuous. 

\begin{corollary}\label{Corollary-3.3}
Under Condition \ref{Condition-2.1} and (\ref{2.2}), the class of functions $\mathbb{I}\{U_{t, j}\leq u_{j}\}-u_{j}
=\mathbb{I}\{F_{j}(G_{j}(\eta_{t}))\leq u_{j}\}-u_{j}$ have Hermite ranks $m_{j}:=\min\{m_{j}(x_{j}):
J_{j(r_{1}, \cdots, r_{p})}(u_{j})\neq 0~\text{for some}~ u_{j}\in\mathbb{R}\}$ with $0<m_{j}(1-2d)<1$.
Then we have,
\begin{equation}\label{3.8}
\mathbb{B}_{\Floor{nz},j}(u_{j})=\frac{\Floor{nz}}{q_{nj}}(D_{\Floor{nz},j}(u_{j})-u_{j})\rightsquigarrow 
\sum_{\sum_{j=1}^{p}r_{j}=m_{j}}\frac{J_{j(r_{1}, \cdots,r_{p})}(u_{j})}{\prod_{j=1}^{p}r_{j}!}
\mathbb{Z}_{j(r_{1}, \cdots,r_{p})}(z).
\end{equation}
for all components $u_{j}\in [0, 1],~j=1, \cdots, p$ and $z\in[0,1]$ in $D([0, 1]^{2})$ of uniformly bounded functions on $[0, 1]^{2}$ 
with supremum norm $\|\cdot\|_{\infty}$. 
\end{corollary}

In order to derive the weak convergence of the sequential empirical copula processes $\mathbb{C}_{\Floor{nz}}$, we first need to study a 
strong approximation of the quantile empirical process $\mathbb{Q}_{\Floor{nz},j}(u_{j})=nq_{nj}^{-1}(D_{\Floor{nz},j}^{-1}(u_{j})-u_{j})$ by 
$\mathbb{B}_{\Floor{nz},j}(u_{j})=nq_{nj}^{-1}(D_{\Floor{nz},j}(u_{j})-u_{j})$ as $n\rightarrow \infty$. 

\begin{proposition}\label{Proposition-3.4}
Suppose that the Condition \ref{Condition-2.1} is satisfied and the subordinated processes $X_{t}$ as in (\ref{2.2}).
Let $J_{j(r_{1}, \cdots,r_{p})}(u_{j})$ and the derivatives $J_{j(r_{1}, \cdots,r_{p})}^{'}(u_{j})$ 
with the Hermite ranks $m_{j}$ are uniformly bounded and $\sup_{u_{j}\in [0, \delta_{n}]}|J_{j(r_{1}, \cdots,r_{p})}(u_{j})|=O(\delta_{n})$ 
for the sequence $\delta_{n}\rightarrow 0$ as $n\rightarrow\infty$. Then we have
\begin{equation}\label{3.9}
\sup_{u_{j}\in[0,1]; z\in[0,1]}\left|\mathbb{Q}_{\Floor{nz},j}(u_{j})-\mathbb{B}_{\Floor{nz},j}(u_{j})\right|
\stackrel{p}\rightarrow 0, \quad \text{as}~ 
n\rightarrow \infty.
\end{equation}
\end{proposition}

\begin{proof}
For each $j\in\{1, \cdots, p\}$, we have
\begin{align*}
\mathbb{Q}_{\Floor{nz},j}(u_{j})&=\frac{\Floor{nz}}{q_{nj}}\left(D_{\Floor{nz},j}^{-1}(u_{j})-u_{j}\right)\\
&=\frac{\Floor{nz}}{q_{nj}}\left(D_{\Floor{nz},j}(D_{\Floor{nz},j}^{-1}(u_{j}))-u_{j}\right)-\frac{\Floor{nz}}{q_{nj}}
\left(D_{\Floor{nz},j}(D_{\Floor{nz},j}^{-1}(u_{j}))-D_{\Floor{nz},j}^{-1}(u_{j})\right)\\
&=\frac{\Floor{nz}}{q_{nj}}\left(D_{\Floor{nz},j}(D_{\Floor{nz},j}^{-1}(u_{j}))-u_{j}\right)-\mathbb{B}_{\Floor{nz},j}
(D_{\Floor{nz},j}^{-1}(u_{j})),
\end{align*}
and 
\begin{equation*}
0\leq \sup_{u_{j}\in[0, 1]}\left|D_{\Floor{nz},j}(D_{\Floor{nz},j}^{-1}(u_{j}))-u_{j}\right|\leq \frac{1}{\Floor{nz}},
\end{equation*}
\allowdisplaybreaks
Thus, we have
\begin{align*}
&\sup_{u_{j}\in[0, 1], z\in[0,1]}|\mathbb{Q}_{\Floor{nz},j}(u_{j})-\mathbb{B}_{\Floor{nz},j}(u_{j})|\\
&\quad=\sup_{u_{j}\in[0, 1],z\in[0,1]}|\mathbb{B}_{\Floor{nz},j}(D_{\Floor{nz},j}^{-1}(u_{j}))-\mathbb{B}_{\Floor{nz},j}(u_{j})|
+O\left(\frac{1}{q_{nj}}\right)\\
&\quad\leq\sup_{u_{j}\in[0, 1],z\in[0,1]}\frac{1}{q_{nj}}\left|\frac{J_{j(r_{1},\cdots, r_{p})}(D_{\Floor{nz},j}^{-1}(u_{j}))}
{\prod_{j=1}^{p}r_{j}!}\sum_{t=1}^{n}H_{r_{1},\cdots, r_{p}}(\eta_{t})\right.\\
&\qquad\left.-\frac{J_{j(r_{1},\cdots, r_{p})}(u_{j})}{\prod_{j=1}^{p}r_{j}!}
\sum_{t=1}^{\Floor{nz}}H_{r_{1},\cdots, r_{p}}(\eta_{t})\right|\\
&\quad\leq \frac{1}{q_{nj}\prod_{j=1}^{p}r_{j}!}\sup_{z\in[0,1]}\left|\sum_{t=1}^{\Floor{nz}}H_{r_{1},\cdots, r_{p}}
(\eta_{t})\right|\sup_{u_{j}\in[0, 1]}
\left|J_{j(r_{1},\cdots, r_{p})}(D_{\Floor{nz},j}^{-1}(u_{j}))-J_{j(r_{1},\cdots, r_{p})}(u_{j})\right|.
\end{align*}
Note that $H_{r_{1},\cdots, r_{p}}(\eta_{t})=\prod_{j=1}^{p}H_{r_{j}}(\eta_{t,j})$. Using the results of \cite{Mori1987}, 
we have
\begin{equation}\label{3.10}
\limsup_{n\rightarrow}n^{m_{j}(d-1/2)}\left(\prod_{j=1}^{p}L_{\eta, j}(n)\log\log n\right)^{-\frac{m_{j}}{2}}
\sup_{z\in[0,1]}\left|\sum_{t=1}^{\Floor{nz}}H_{r}(\eta_{t})\right|=\frac{2^{(m_{j}+1)/2}}{\sqrt{c_{j}(r, d)}}
\end{equation}
almost surely. Then, by the mean value theorem, we derive
\begin{align*}
&\sup_{u_{j}\in[0, 1]}\left|J_{j(r_{1},\cdots, r_{p})}(D_{\Floor{nz},j}^{-1}(u_{j}))-J_{j(r_{1},\cdots, r_{p})}(u_{j})\right|\\
&\quad=\sup_{u_{j}\in[0, 1]}|D_{\Floor{nz},j}^{-1}(u_{j})-u_{j})||J_{j(r_{1},\cdots, r_{p})}^{'}(\widetilde{D}_{\Floor{nz},j}(u_{j}))|
\end{align*}
where $|\widetilde{D}_{\Floor{nz},j}^{-1}(u_{j})-u_{j})|\leq |D_{\Floor{nz},j}^{-1}(u_{j})-u_{j})|$ and 
$|J_{j(r_{1},\cdots, r_{p})}^{'}(\widetilde{D}_{\Floor{nz},j}(u_{j}))|$ is uniformly bounded.
Then, from the proof (\ref{3.10}) of Proposition 2.2 in \cite{Csrg2006}, we obtain
\begin{align*}
\sup_{u_{j}\in[0, 1]}\left|D_{\Floor{nz},j}^{-1}(u_{j})-u_{j})\right|=\sup_{u_{j}\in[0, 1]}\frac{q_{nj}}{\Floor{nz}}\left|
\mathbb{Q}_{\Floor{nz},j}(u_{j})\right|\\
=O\left(\left(n^{2d-1}\prod_{j=1}^{p}L_{\eta, j}(n)\log\log n\right)^{\frac{m_{j}}{2}}\right)\rightarrow 0.
\end{align*}
as $n\rightarrow \infty$. This yields 
\begin{align}\label{3.11}
&\sup_{u_{j}\in[0, 1]}\left|J_{j(r_{1},\cdots, r_{p})}(D_{\Floor{nz},j}^{-1}(u_{j}))-J_{j(r_{1},\cdots, r_{p})}(u_{j})\right|\nonumber\\
&\quad=O\left(\left(n^{2d-1}\prod_{j=1}^{p}L_{\eta, j}(n)\log\log n\right)^{\frac{m_{j}}{2}}\right)\rightarrow 0.
\end{align}
as $n\rightarrow \infty$. Therefore, by combining (\ref{3.10}) and (\ref{3.11}), we conclude that
\begin{equation*}
\sup_{u_{j}\in[0, 1]}\left|\mathbb{Q}_{\Floor{nz},j}(u_{j})-\mathbb{B}_{\Floor{nz},j}(u_{j})\right|\stackrel{p}\rightarrow 0,
\end{equation*}
as $n\rightarrow \infty$, which complete the proof of proposition.
\end{proof}

\section{Weak convergence of the sequential empirical copula processes}

\subsection{Weak convergence of the processes in (\ref{2.5})}

\hspace{0.5cm}
In this subsection, we study the asymptotic behavior of the sequential empirical copula processes $D_{\Floor{nz}}$.
In the case of long-range dependent sequences, the limiting distributions of the $D_{\Floor{nz}}$ 
are usually dependent on the Hermite rank $r$ of the class of functions $\mathbb{I}\{\mathbf{U}_{t}\leq\mathbf{u}\}-C(\mathbf{u})$,
and determined by functions $G_{1},\cdots, G_{p}$. Thus, we show first the asymptotics of $D_{\Floor{nz}}$ under lower Hermite ranks.

\begin{lemma}\label{Lemma-4.1}
If the Condition \ref{Condition-2.1} and (\ref{2.2}) holds true, the class of functions
$\mathbb{I}\{\mathbf{U}_{t}\leq\mathbf{u}\}-C(\mathbf{u})$ has a Hermite rank $r=1$ and $0<r(1-2d)<1$, then
\begin{equation*}
\mathbb{B}_{\Floor{nz}}(\mathbf{u})=\frac{\Floor{nz}}{q_{n}}\{D_{\Floor{nz}}(\mathbf{u})-C(\mathbf{u})\}\rightsquigarrow 
\mathbb{B}_{C}(\mathbf{u})
\end{equation*}
weakly converges to 
\begin{equation*}
\mathbb{B}_{C}(\mathbf{u})=-\phi(\Phi^{-1}(u_{1}))\mathbb{H}_{1, 0,\cdots, 0}(z)
-\left(\sum_{j=2}^{p}\phi(C_{j|1, \cdots, j-1}(u_{j})\right)\mathbb{H}_{1}^{(j)}(z)
\end{equation*}
with covariance function
\begin{equation}\label{4.1}
\mathrm{Cov}(\mathbb{H}_{1}^{(j)}(z_{1}), \mathbb{H}_{1}^{(j)}(z_{2}))=
\frac{1}{2}\left\{|z_{1}|^{1+2d}+|z_{2}|^{1+2d}-|z_{1}-z_{2}|^{1+2d}\right\},
\end{equation}
where the processes $\mathbb{H}_{1}^{(j)}(z)$ denotes the $j$th independent copies 
of fractional Brownian motions. 
\end{lemma}

\begin{proof}
For simplicity of proof, we discuss here only with the bivariate case. Let us first compute the Hermite coefficients 
$J_{1, 0}(\mathbf{u})$ and $J_{0,1}(\mathbf{u})$. Let 
\begin{equation*}
G_{1}(\eta_{1})=\eta_{1}\quad\text{and}\quad G_{2}(\eta_{1}, \eta_{2})=C_{2|1}^{-1}(F_{2}(\eta_{2})|\eta_{1}))
\end{equation*}
in $L^{2}(\mathbb{R}^{2}, \phi(\eta_{1})\phi(\eta_{2}))$, where $C_{2|1}^{-1}(F_{2}(\eta_{2})|\eta_{1})$ is monotonically
nondecreasing for given $\eta_{1}$ and strictly increasing on the set of Lebesgue measure and denote as $C_{2|1}^{-1}$.
By using the formula (\ref{2.9}), we can calculate the Hermite coefficient
\allowdisplaybreaks
\begin{align*}
J_{1, 0}(u_{1}, u_{2})&=\mathbf{E}[\mathbb{I}(U_{t,1}\leq u_{1}, U_{t,2}\leq u_{2})-C(u_{1}, u_{2})]H_{1}(\eta_{1})\\
&=\int_{\mathbb{R}^{2}}\mathbb{I}(F_{1}(G_{1}(\eta_{1}))\leq u_{1}, F_{2}(G_{2}(\eta_{2},\eta_{1}))\leq u_{2})\eta_{1}
\phi(\eta_{1})\phi(\eta_{2})\mathrm{d}\eta_{1}\mathrm{d}\eta_{2}\\
&=\int_{\mathbb{R}^{2}}\mathbb{I}(\eta_{1}\leq \Phi^{-1}(u_{1}))\mathbb{I}(C_{2|1}^{-1}(F_{2}(\eta_{2})|\eta_{1})
\leq u_{2})\eta_{1}\phi(\eta_{1})\phi(\eta_{2})\mathrm{d}\eta_{1}\mathrm{d}\eta_{2}\\
&=\int_{\mathbb{R}}\mathbb{I}(\eta_{1}\leq \Phi^{-1}(u_{1}))\left[\int_{\mathbb{R}}\mathbb{I}\left(
C_{2|1}^{-1}(F_{2}(\eta_{2})|\eta_{1})\leq u_{2}\right)\phi(\eta_{2})\mathrm{d}\eta_{2}\right]\eta_{1}
\phi(\eta_{1})\mathrm{d}\eta_{1}\\
&=\int_{0}^{\Phi^{-1}(u_{1})}\eta_{1}\phi(\eta_{1})\mathrm{d}\eta_{1}=-\phi(\Phi^{-1}(u_{1})).
\end{align*}
is nonzero, where $H_{1}(\eta_{1})=\eta_{1}$ and the limit distribution in the right hand side of (4.9) is $-\phi(\Phi^{-1}(u_{1}))
\mathbb{H}_{1, 0}(1)$ and the Hermite rank is one. The process $\mathbb{H}_{1,0}(1)$ is a Gaussian with mean zero and unit variance. 
Furthermore, we also calculate
\allowdisplaybreaks
\begin{align*}
J_{0, 1}(u_{1}, u_{2})&=\mathbf{E}[\mathbb{I}(U_{t, 1}\leq u_{1}, U_{t, 2}\leq u_{2})-C(u_{1}, u_{2})]H_{1}(\eta_{2})\\
&=\int_{\mathbb{R}^{2}}\mathbb{I}(F_{1}(G_{1}(\eta))\leq u_{1}, F_{2}(G_{2}(\eta_{2},\eta_{1}))\leq u_{2})\eta_{2}
\phi(\eta_{1})\phi(\eta_{2})\mathrm{d}\eta_{1}\mathrm{d}\eta_{2}\\
&=\int_{\mathbb{R}^{2}}\mathbb{I}(\eta_{1}\leq \Phi^{-1}(u_{1}))\mathbb{I}(C_{2|1}^{-1}(F_{2}(\eta_{2})|\eta_{1})\leq u_{2})\eta_{2} 
\phi(\eta_{1})\phi(\eta_{2})\mathrm{d}\eta_{1}\mathrm{d}\eta_{2}\\
&=\int_{\mathbb{R}}\mathbb{I}(\eta\leq \Phi^{-1}(u_{1}))\left[\int_{\mathbb{R}}\mathbb{I}\left(
C_{2|1}^{-1}(F_{2}(\eta_{2})|\eta_{1})\leq u_{2}\right)\eta_{2}\phi(\eta_{2})\mathrm{d}\eta_{2}\right]\phi(\eta_{1})\mathrm{d}\eta_{1}\\
&=-\phi(C_{2|1}(u_{2}))\int_{\mathbb{R}}\mathbb{I}(\eta_{1}\leq \Phi^{-1}(u_{1}))\phi(\eta_{1})\mathrm{d}\eta_{1}\\
&=-\phi(C_{2|1}(u_{2})).
\end{align*}
is also nonzero. Thus, by the results of \cite{Taqqu1975}, the limiting processes $\mathbb{H}_{1, 0}(1)$ and $\mathbb{H}_{0, 1}(1)$
are fractional Brownian motions with mean zero, stationary increments as in (\ref{4.1}). 
Consequently, by induction method, this result can be extended to the multivariate cases. 
\end{proof}

Now we will establish a limiting distribution of empirical copula process $\mathbb{B}_{\Floor{nz}}$ with Hermite rank $r=2$ case. The 
limiting process of $\mathbb{H}_{2}(z)$ are more difficult than $\mathbb{H}_{1}(z)$: 

\begin{lemma}\label{Lemma-4.2}
If the Condition \ref{Condition-2.1} and (\ref{2.2}) holds true, the class of functions
$\mathbb{I}\{\mathbf{U}_{t}\leq \mathbf{u}\}-\mathbf{u}$ has Hermite rank $r=2$ and $0<r(1-2d)<1$, then
\begin{align}\label{4.2}
\mathbb{B}_{\Floor{nz}}(\mathbf{u})&=\frac{\Floor{nz}}{q_{n}}\{D_{\Floor{nz}}(\mathbf{u})-C(\mathbf{u})\}\rightsquigarrow
-\frac{\sqrt{2F_{1}^{-1}(u_{1})}}{2!\sqrt{\pi}}e^{-\frac{F_{1}^{-1}(u_{1})}{2}}\mathbb{H}_{2, 0\cdots, 0}(z)\nonumber\\
&\quad-\sum_{j=2}^{p}\left(\frac{\sqrt{2C_{j|1, \cdots, j-1}(F_{j}^{-1}(u_{j}))}}{2!\sqrt{\pi}}
e^{-\frac{C_{j|1, \cdots, j-1}(F_{j}^{-1}(u_{j}))}{2}}\right)\mathbb{H}_{2}^{(j)}(z),
\end{align}
where the process $\mathbb{H}_{2, 0, \cdots, 0}^{j}(z)$
are the $j$th independent copies of Rosenblatt process, which are non-Gaussian and has stationary increments, the covariance
function of $\mathbb{H}_{r_{1}, \cdots, r_{p}}(z)$ is given by
\begin{equation}\label{4.3}
\mathrm{Cov}(\mathbb{H}_{2}^{j}(z_{1}), \mathbb{H}_{2}^{j}(z_{2}))=
\frac{1}{2}\left\{|z_{1}|^{4d}+|z_{2}|^{4d}-|z_{1}-z_{2}|^{4d}\right\}.
\end{equation}
\end{lemma}

\begin{proof}
Since the limiting processes of $D_{\Floor{nz}}(\mathbf{u})$ are usually dependent on the Hermite rank $r$ of 
the class of functions $\mathbb{I}(\mathbf{U}_{t}\leq \mathbf{u})-C(\mathbf{u})$ and it is determined by the functions $G_{1}$ and
$G_{2}$, we need to compute the Hermite coefficients.

Consider the function $G_{1}(\eta_{1})=\eta_{1}^{2}$ and $G_{2}(\eta_{1}, \eta_{2})=\left[C_{2|1}^{-1}(F_{2}(\eta_{2})|\eta_{1})\right]^{2}$ 
and $H_{2}(\eta_{j})=\eta_{j}^{2}-1$ for $j=1, 2$.
Thus, we calculate the Hermite coefficients $J_{2, 0}, J_{0,2}$ and $J_{1,1}$ respectively.
\allowdisplaybreaks
\begin{align*}
J_{1, 0}(u_{1}, u_{2})&=\mathbf{E}[\mathbb{I}(U_{1}\leq u_{1}, U_{2}\leq u_{2})-C(u_{1}, u_{2})]H_{1}(\eta_{1})\\
&=\mathbf{E}\left[\mathbb{I}\{F_{1}(G_{1}(\eta_{1}))\leq u_{1}, F_{2}(G_{2}(\eta_{1}, \eta_{2}))\leq u_{2}\}\right]\eta_{1}\\
&=\mathbf{E}\left[\mathbb{I}\left(-\sqrt{F_{1}^{-1}(u_{1})}\leq\eta_{1}\leq \sqrt{F_{1}^{-1}(u_{1})}\right.\right.\\
&\quad\left.\left.-\sqrt{F_{2}^{-1}(u_{2})}\leq C_{2|1}^{-1}(F_{2}(\eta_{2})|\eta_{1})\leq \sqrt{F_{2}^{-1}(u_{2})}\right)\right]\eta_{1}\\
&=\int_{\mathbb{R}}\left[\mathbb{I}\left(-\sqrt{F_{1}^{-1}(u_{1})}\leq\eta_{1}\leq \sqrt{F_{1}^{-1}(u_{1})}\right)\right]
\eta_{1}\phi(\eta_{1})\mathrm{d}\eta_{1}\\
&\quad\times\int_{\mathbb{R}}\left[\mathbb{I}\left(-\sqrt{F_{2}^{-1}(u_{2})}\leq C_{2|1}^{-1}(F_{2}(\eta_{2})|\eta_{1})\leq 
\sqrt{F_{2}^{-1}(u_{2})}\right)\right]\phi(\eta_{2})\mathrm{d}\eta_{2}=0.
\end{align*}
By symmetricity of $J_{1, 0}$ and $J_{0,1}$, we obtain $J_{0, 1}(u_{1}, u_{2})=0$. If $r=2$, then we have
\allowdisplaybreaks
\begin{align*}
J_{2, 0}(u_{1}, u_{2})&=\mathbf{E}[\mathbb{I}(U_{1}\leq u_{1}, U_{2}\leq u_{2})-C(u_{1}, u_{2})]H_{2}(\eta_{1})\\
&=\mathbf{E}[\mathbb{I}(U_{1}\leq u_{1}, U_{2}\leq u_{2})-C(u_{1}, u_{2})](\eta_{1}^{2}-1)\\
&=\int_{\mathbb{R}}\left[\mathbb{I}\left(-\sqrt{F_{1}^{-1}(u_{1})}\leq\eta_{1}\leq \sqrt{F_{1}^{-1}(u_{1})}\right)\right]
(\eta_{1}^{2}-1)\phi(\eta_{1})\mathrm{d}\eta_{1}\\
&\quad\times\int_{\mathbb{R}}\left[\mathbb{I}\left(-\sqrt{F_{2}^{-1}(u_{2})}\leq C_{2|1}^{-1}(F_{2}(\eta_{2})|\eta_{1})\leq 
\sqrt{F_{2}^{-1}(u_{2})}\right)\right]\phi(\eta_{2})\mathrm{d}\eta_{2}\\
&=\int_{-\sqrt{F_{1}^{-1}(u_{1})}}^{\sqrt{F_{1}^{-1}(u_{1})}}(\eta_{1}^{2}-1)\phi(\eta_{1})\mathrm{d}\eta_{1}
=-\frac{\sqrt{2F_{1}^{-1}(u_{1})}}{\sqrt{\pi}}e^{-\frac{F_{1}^{-1}(u_{1})}{2}}.
\end{align*}
Note that $x_{1}=F_{1}^{-1}(u_{1})$. By symmetricity of $J_{0,2}(\mathbf{u})$ and $J_{2,0}(\mathbf{u})$, we can easly derive 
\allowdisplaybreaks
\begin{align*}
J_{0, 2}(u_{1}, u_{2})&=\mathbf{E}[\mathbb{I}(U_{1}\leq u_{1}, U_{2}\leq u_{2})-C(u_{1}, u_{2})]H_{2}(\eta_{2})\\
&=\mathbf{E}[\mathbb{I}(U_{1}\leq u_{1}, U_{2}\leq u_{2})-C(u_{1}, u_{2})](\eta_{2}^{2}-1)\\
&=\int_{\mathbb{R}}\left[\mathbb{I}\left(-\sqrt{F_{1}^{-1}(u_{1})}\leq\eta_{1}\leq \sqrt{F_{1}^{-1}(u_{1})}\right)\right]
\phi(\eta_{1})\mathrm{d}\eta_{1}\\
&\quad\times\int_{\mathbb{R}}\left[\mathbb{I}\left(-\sqrt{F_{2}^{-1}(u_{2})}\leq C_{2|1}^{-1}(F_{2}(\eta_{2})|\eta_{1})\leq 
\sqrt{F_{2}^{-1}(u_{2})}\right)\right](\eta_{2}^{2}-1)\phi(\eta_{2})\mathrm{d}\eta_{2}\\
&=\int_{C_{2|1}(-\sqrt{F_{2}^{-1}(u_{2})}|\eta_{1})}^{C_{2|1}(\sqrt{F_{2}^{-1}(u_{2})}|\eta_{1})}
(\eta_{2}^{2}-1)\phi(\eta_{2})\mathrm{d}\eta_{2}\\
&=-\frac{\sqrt{2C_{2|1}(F_{1}^{-1}(u_{2})|\eta_{1}})}{\sqrt{\pi}}e^{-\frac{C_{2|1}(F_{1}^{-1}(u_{2})|\eta_{1})}{2}}.
\end{align*}
Thus, limiting processes $\mathbb{H}_{2, 0}(z), \mathbb{H}_{0,2}(z)$ are called Rosenblatt process and Hermite rank $r=2$,
it is not difficult to see that $\mathbb{H}_{2,0}(z), \mathbb{H}_{0,2}(z)$ are non-Gaussian, and has stationary increments
with covariance
\begin{equation*}
\mbox{Cov}(\mathbb{H}_{2, 0}(z_{1}), \mathbb{H}_{2,0}(z_{2}))=\frac{1}{2}
\left\{|z_{1}|^{4d}+|z_{2}|^{4d}-|z_{1}-z_{2}|^{4d}\right\}.
\end{equation*}
However, we need to calculate
\allowdisplaybreaks
\begin{align*}
J_{1,1}(u_{1}, u_{2})&=\mathbf{E}[\mathbb{I}(F_{1}(G_{1}(\eta_{1}))\leq u_{1}, F_{2}(G_{2}(\eta_{1}, \eta_{2}))\leq u_{2})]
H_{1}(\eta_{1})H_{1}(\eta_{2})\\
&=\int_{\mathbb{R}}\left[\mathbb{I}\left(-\sqrt{F_{1}^{-1}(u_{1})}\leq\eta_{1}\leq \sqrt{F_{1}^{-1}(u_{1})}\right)\right]
\eta_{1}\phi(\eta_{1})\mathrm{d}\eta_{1}\\
&\quad\times\int_{\mathbb{R}}\left[\mathbb{I}\left(-\sqrt{F_{2}^{-1}(u_{2})}\leq C_{2|1}^{-1}(F_{2}(\eta_{2})|\eta_{1})\leq 
\sqrt{F_{2}^{-1}(u_{2})}\right)\right]\eta_{2}\phi(\eta_{2})\mathrm{d}\eta_{2}\\
&=\int_{-\sqrt{F_{1}^{-1}(u_{1})}}^{\sqrt{F_{1}^{-1}(u_{1})}}\eta_{1}\phi(\eta_{1})\mathrm{d}\eta_{1}
\int_{C_{2|1}(-\sqrt{F_{2}^{-1}(u_{2})}|\eta_{1})}^{C_{2|1}(\sqrt{F_{2}^{-1}(u_{2})}|\eta_{1})}\eta_{2}
\phi(\eta_{2})\mathrm{d}\eta_{2}\\
&=0.
\end{align*}
By the results of \cite{Taqqu1975}, we can now conclude that the limiting distributions $\mathbb{H}_{2, 0}(z)$ and $\mathbb{H}_{0,2}(z)$ of 
$D_{\Floor{nz}}(\mathbf{u})$ are no longer Gaussian, which are called the Rosenblatt process with stationary increments as described
in (\ref{4.3}). By using the induction, we can conclude that (\ref{4.2}) holds true. 
\end{proof}

Furthermore, the limiting distributions of the process $\mathbb{B}_{\Floor{nz}}$ are no longer a Gaussian process 
(which are called Hermite processes) if the class of functions $\mathbb{I}\{\mathbf{U}_{t}\leq \mathbf{u}\}-C(\mathbf{u})$
has Hermite rank $r>2$. We will summarize this resul in the following theorem.

\begin{theorem}\label{Theorem-4.3}
Let $(\eta_{t,j})_{t\in\mathbb{Z}}$ be stationary Gaussian processes satisfying Condition \ref{Condition-2.1} and \ref{Condition-4.1}.
Let the subordinated process $X_{t, j}=G_{j}(\eta_{t})$ as in (\ref{2.2}). Then
the class of functions $\{\mathbb{I}\{\mathbf{U}_{t}\leq \mathbf{u}\}-C(\mathbf{u})$ have Hermite rank $\sum_{j=1}^{p}r_{j}=r>2$ 
and satisfies $0<\sum_{j=1}^{p}r_{j}(1-2d_{j})<1$ for $d_{j}\neq d\in(0, 1/2)$ such that
\begin{equation}\label{4.4}
\mathbb{B}_{\Floor{nz}}(\mathbf{u})=\frac{\Floor{nz}}{q_{n}(r_{11}^{*}\cdots r_{1p}^{*})}\{D_{\Floor{nz}}(\mathbf{u})-C(\mathbf{u})\}
\rightsquigarrow\sum_{i=1}^{M}\rho_{i}\frac{J_{r_{11}^{*}\cdots r_{1p}^{*}}(\mathbf{u})}
{r_{11}^{*}!\cdots r_{1p}^{*}!}\mathbb{H}_{r_{11}^{*}\cdots r_{1p}^{*}}(z),
\end{equation}
and If the all memory parameters $d_{j}=d\in(0, 1/2)$, then
\begin{equation}\label{4.5}
\mathbb{B}_{\Floor{nz}}(\mathbf{u})=\frac{\Floor{nz}}{q_{n}(r_{1}\cdots r_{1})}\{D_{\Floor{nz}}(\mathbf{u})-C(\mathbf{u})\}
\rightsquigarrow \sum_{\sum_{j=1}^{p}r_{j}=r}\frac{J_{r_{1}\cdots r_{1}}(\mathbf{u})}{r_{1}!\cdots r_{p}!}\mathbb{H}_{r_{1}\cdots r_{1}}(z),
\end{equation}
for all $\mathbf{u}\in[0,1]^{p}$ and $z\in[0, 1]$ in the space $l^{\infty}([0,1]^{p+1})$ of uniformly bounded functions on $[0,1]^{p+1}$
equipped with supremum norm $\|\cdot\|_{\infty}$, where the Hermite process $\mathbb{H}_{r_{1}\cdots r_{p}}(z)$ is given in (\ref{2.10}). 
\end{theorem}

\begin{proof}
The asymptotic behavior of the empirical processes $\mathbb{B}_{\Floor{nz}}(\mathbf{u})$ can be established
based on the multivariate reduction principle obtained by \cite{Mounirou2016} for $d_{j}\neq d$ for $j=1, \cdots, p$.
or the results of Theorem 9 in \cite{Arcones1994} for $d_{j}=d$ for $j=1, \cdots, p$ respectively. 
Since $X_{t,j}\leq F_{\Floor{nz},j}^{-1}(u_{j})\leq u_{j}$ if and only if $U_{t,j}\leq D_{\Floor{nz},j}^{-1}(u_{j})$ for all
$t\in\{1, \cdots, n\}, j\in \{1, \cdots, p\}$. Without loss of generality, we can assume that $\mathbf{X}_{t}$ 
and $\mathbf{U}_{t}$ has copula $C$ such that the multivariate expansion of $D_{\Floor{nz}}(\mathbf{u})-C(\mathbf{u})$ in Hermite 
polynomials is written as
\begin{align*}
&\frac{\Floor{nz}}{q_{n}(r_{11}^{*}\cdots r_{1p}^{*})}\{D_{\Floor{nz}}(\mathbf{u})-C(\mathbf{u})\}\\
&\quad=\frac{1}{q_{n}(r_{11}^{*}\cdots r_{1p}^{*})}
\sum_{i=1}^{M}\frac{J_{r_{i1}^{*}\cdots r_{ip}^{*}}(\mathbf{u})}{r_{i1}^{*}!\cdots r_{ip}^{*}!}\sum_{t=1}^{\Floor{nz}}
H_{r_{i1}^{*}}(\eta_{t,1})\cdots H_{r_{ip}}^{*}(\eta_{t,p}),
\end{align*}
then based on the result of Proposition 1 and 2 in \cite{Mounirou2016}, we derive
\begin{align*}
\frac{1}{q_{n}(r_{11}^{*}\cdots r_{1p}^{*})}\sum_{t=1}^{\Floor{nz}}H_{r_{i1}^{*}}(\eta_{t,1})\cdots H_{r_{ip}}^{*}(\eta_{t,p})
\rightsquigarrow \mathbb{H}_{r_{i1}^{*}\cdots r_{ip}^{*}}(z),
\end{align*}
in the unifrom space $D([0, 1]^{p})$, Thus we derive (\ref{4.4}).
If the memory parameters $d_{j}=d, ~j=1, \cdots, p$ are all equal, then (\ref{4.5}) can be derived based on the Theorem 9 in
\cite{Arcones1994}. However, they prove their results by using a tightness condition on the empirical process rather than a uniform
reduction principle. 
\end{proof}

\subsection{Weak convergence of the processes in (\ref{2.6})}

\hspace{0.5cm}
In this section, we first study the aysmptotic behavior of the empirical processes $\mathbb{B}_{\Floor{nz}}$ under the 
long-range dependent sequences, where the mempory parameters $d_{j}$ are all equal. Then, 
to establish the weak convergence of the empirical copula process $\mathbb{C}_{\Floor{nz}}$ under long-range dependence,
a smoothness condition on copula $C$ as point out by \cite{Segers2012} is still needed.
This condition is very useful to establish the weak convergence of empirical copula process $\mathbb{C}_{\Floor{nz}}$ on the 
boundaries as well as important for the limiting process of $\mathbb{C}_{\Floor{nz}}$ to exist and to have continuous trajectories. 

\begin{condition}\label{Condition-4.4}
For $j=1,\cdots, p$, the $j$th first-order partial derivative $\dot{C}_{j}$ exists and is continuous
on the set $\{\mathbf{u}\in[0,1]^{p}|~ u_{j}\in(0,1)\}$ for all $j=1, \cdots, p$.
\end{condition}

Under Condition \ref{Condition-4.4}, the partial derivatives $C_{j}$ can be defined on whole unit cube $[0,1]^{p}$ by
\begin{align}\label{4.6}
\dot{C}_{j}(\mathbf{u})&=\begin{cases}
\limsup\limits_{h\rightarrow 0}\frac{C(\mathbf{u}+h\mathbf{e}_{j})-C(\mathbf{u})}{h}, &\mbox{for}~ 0<u_{j}<1,\\
\limsup\limits_{h\downarrow 0}\frac{C(\mathbf{u}+h\mathbf{e}_{j})}{h}, &\mbox{for}~ u_{j}=0,\\
\limsup\limits_{h\rightarrow 0}\frac{C(\mathbf{u})-C(\mathbf{u}-h\mathbf{e}_{j})}{h}, &\mbox{for}~ u_{j}=1,
\end{cases}
\end{align}
where $\mathbf{e}_{1}$ denotes the first column of a $2\times 2$ matrix and $\mathbf{u}=(u_{1}, \cdots, u_{p})^{'}\in[0,1]^{p}$. 
It can be seen from (\ref{4.6}) that it expand the application of the many copula families, see, e.g., \cite{Segers2012}.

The asymptotic behavior of the empirical copula process under i.i.d and weakly dependent cases has been established in number of 
literature \cite{Ruschendorf1976}, \cite{Fermanian2004},\cite{Doukhan2008},\cite{Segers2012},\cite{BucherRuppert2013} and \cite{Bucher2013}. 
Here we study the limiting distributions of the sequential empirical copula processes with the Hermite rank $r\geq 1$. 
If the class of functions $\mathbb{I}\{\widehat{\mathbf{U}}_{t}\leq \mathbf{u}\}-C(\mathbf{u})$ with 
$\widehat{\mathbf{U}}_{t}=(\widehat{U}_{t,1}, \cdots, \widehat{U}_{t,p})^{'}$ for all $\mathbf{u}\in[0,1]^{p}$ 
has Hermite rank $r\geq 1$, then the limiting distributions converges to Hermite processes. Our main results are summarized
in the following theorem:

\begin{theorem}\label{Theorem-4.5}
Let $(\eta_{t,j})_{t\in\mathbb{Z}}$ be stationary Gaussian processes satisfying Condition \ref{Condition-2.1} and \ref{Condition-4.1}.
Let the subordinated process $X_{t, j}=G_{j}(\eta_{t})$ as in (\ref{2.2}). Then
the class of functions $\mathbb{I}\{\widehat{\mathbf{U}}_{t}\leq \mathbf{u}\}-C(\mathbf{u})$ have Hermite rank $r$
and satisfies $0<\sum_{j=1}^{p}r_{j}(1-2d_{j})<1$. Moreover, let $J_{r_{1},\cdots, r_{p}}(\mathbf{u})$ and the derivatives
$J_{r_{1},\cdots, r_{p}}^{'}(\mathbf{u})$ are uniformly bounded and $\sup_{\mathbf{u}\in[0, \delta_{n}]}
|J_{r_{1},\cdots, r_{p}}^{'}(\mathbf{u})|=O(\delta_{n})$. If copula $C$ satisfies the Condition \ref{Condition-4.4}
 and the memory parameters are same $d_{j}=d\in (0, 1/2)$, then
\begin{equation*}
\mathbb{C}_{\Floor{nz}}(\mathbf{u})=\frac{\Floor{nz}}{q_{n}(r_{1}\cdots r_{p})}\{C_{\Floor{nz}}(\mathbf{u})-C(\mathbf{u})\}
\end{equation*}
converges weakly to
\begin{align}\label{4.7}
\mathbb{C}_{z}(\mathbf{u})&=\sum_{\sum_{j=1}^{p}r_{j}=r}\frac{J_{r_{1}\cdots r_{p}}(\mathbf{u})}{r_{1}
!\cdots r_{p}!}\mathbb{H}_{r_{1}\cdots r_{p}}(z)-\dot{C}_{1}(\mathbf{u})\frac{J_{r_{1}}(u_{1})}{r_{1}!}
\mathbb{Z}_{r_{1}}(z)\nonumber\\
&-\dot{C}_{2}(\mathbf{u})\sum_{r_{1}+r_{2}=m_{2}}^{\infty}\frac{J_{r_{1}, r_{2}}(u_{2})}{r_{1}!r_{2}!}\mathbb{Z}_{r_{1}, r_{2}}(z)
\cdots -\dot{C}_{p}(\mathbf{u})\sum_{r_{1}+\cdots+r_{p}=m_{p}}^{\infty}\frac{J_{r_{1}\cdots r_{p}}(u_{p})}
{r_{!}\cdots !r_{p}!}\mathbb{Z}_{r_{1}\cdots r_{p}}(z)
\end{align}
for all $\mathbf{u}\in [0, 1]^{p}, z\in[0, 1]$ in $l^{\infty}([0,1]^{p+1})$, where $\dot{C}_{j}$ denotes the
$j$-th partial derivatives of $C$. 
\end{theorem}

\begin{proof} The weak convergence of $\mathbb{C}_{\Floor{nz}}$ can be achieved by two different methods. One of the methods is
functional delta method, which was applied in \cite{Fermanian2004} and \cite{Bucher2011} for i.i.d 
case respectively. The another one is Seger's method introduced in \cite{Segers2012}. These both methods are still valid for the case
of long-memory sequences to derive the weak convergence of the sequential empirical copula processes $\mathbb{C}_{\Floor{nz}}$. 
Here we prefer to establish the weak convergence of $\mathbb{C}_{\Floor{nz}}$ by using Segers' approach. 

To show the weak convergence $\mathbb{C}_{\Floor{nz}}\rightsquigarrow \mathbb{C}_{z}$, it suffices to prove for 
$\mathbf{u}\in[0,1]^{p}$: 
\begin{equation*}
\sup_{z\in[0,1],\mathbf{u}\in[0,1]^{p}}\left|\widehat{\mathbb{C}}_{\Floor{nz}}(\mathbf{u}-\mathbb{C}_{\Floor{nz}}(\mathbf{u})\right|
\stackrel{p}{\rightarrow} 0, \quad\text{as}~ n\rightarrow\infty,
\end{equation*}
where 
\begin{align*}
\widehat{\mathbb{C}}_{\Floor{nz}}(\mathbf{u})&=\mathbb{B}_{\Floor{nz}}(\mathbf{u})-\sum_{j=1}^{p}\dot{C}_{j}(\mathbf{u})
\mathbb{B}_{\Floor{nz}, j}(u_{j})
\end{align*}
is the sequence of processes of $\mathbb{C}_{z}(\mathbf{u})$ in $\ell^{\infty}([0,1]^{p+1})$. The supremum is zero if $u_{j}=0$ 
for some $j=1, \cdots, p$. For simpilicity, we denote the quantile marginal sequential empirical distributions as
\begin{equation*}
v_{\Floor{nz}}(\mathbf{u})=(D_{\Floor{nz},1}^{-1}(u_{1}), \cdots, D_{\Floor{nz}, 2}^{-1}(u_{p}))\quad 
\text{for all}~ \mathbf{u}\in[0,1]^{p}, z\in[0,1]
\end{equation*}
then we decompose the empirical copula process $\mathbb{C}_{\Floor{nz}}$ by
\begin{align}\label{4.8}
\mathbb{C}_{\Floor{nz}}(\mathbf{u})&=\frac{[nz]}{q_{n}(r_{1}\cdots r_{p})}\{C_{\Floor{nz}}(\mathbf{u})-C(\mathbf{u})\}\nonumber\\
&=\frac{\Floor{nz}}{q_{n}(r_{1}\cdots r_{p})}\{D_{\Floor{nz}}(v_{n}(\mathbf{u}))-C(v_{\Floor{nz}}(\mathbf{u}))\}
+\frac{\Floor{nz}}{q_{n}(r_{1}\cdots r_{p})}\{C(v_{\Floor{nz}}(\mathbf{u}))-C(\mathbf{u})\}\nonumber\\
&=\mathbb{B}_{\Floor{nz}}(v_{\Floor{nz}}(\mathbf{u}))+\frac{[nz]}{q_{n}(r_{1}\cdots r_{p})}\{C(v_{\Floor{nz}}(\mathbf{u}))-C(\mathbf{u})\}.
\end{align}
Note that the normalizing factor $q_{n}(r_{11}^{*}\cdots r_{1p}^{*})$ can be reduced to $q_{n}(r_{1}\cdots r_{p})$ if the memory
parameters $d_{j}=d$ are all equal. As we proved above, $\mathbb{B}_{\Floor{nz}}(\mathbf{u})$ converges weakly to 
the process in (\ref{3.3}) in a metric space $l^{\infty}([0,1]^{p+1})$ and have continuous trajectories. From the proof of Proposition 
\ref{Proposition-3.4}, we have
\begin{align*}
\sup_{z\in[0,1], u_{j}\in[0,1]}|D_{\Floor{nz},j}^{-1}(u_{j})-u_{j}|\rightarrow 0, \quad\text{a.s.}
\end{align*}
Then we show that
\allowdisplaybreaks
\begin{align}\label{4.9}
&\sup_{z\in[0,1], \mathbf{u}\in[0,1]^{p}}\left|\mathbb{B}_{\Floor{nz}}(v_{\Floor{nz}}(\mathbf{u}))-\mathbb{B}_{\Floor{nz}}(\mathbf{u})
\right|\nonumber\\
&\quad\leq\left|\frac{\Floor{nz}}{q_{n}(r_{1}\cdots r_{p})}\{D_{\Floor{nz}}(v_{\Floor{nz}}(\mathbf{u}))
-D_{\Floor{nz}}(\mathbf{u}))\}\right|\nonumber\\
&\quad=\sup_{z\in[0,1], \mathbf{u}\in[0,1]^{p}}\frac{1}{q_{n}(r_{1},\cdots, r_{p})}
\left|\frac{J_{r_{1},\cdots, r_{p}}(v_{\Floor{nz}}(\mathbf{u}))}{\prod_{j=1}^{p}r_{j}!}\sum_{t=1}^{\Floor{nz}}
H_{r_{1},\cdots, r_{p}}(\eta_{t})\right.\nonumber\\
&\qquad\left.-\frac{J_{r_{1},\cdots, r_{p}}(\mathbf{u})}{\prod_{j=1}^{p}r_{j}!}\sum_{t=1}^{\Floor{nz}}
H_{r_{1},\cdots, r_{p}}(\eta_{t})\right|\nonumber\\
&\quad\leq\frac{1}{q_{n}(r_{1},\cdots, r_{p})\prod_{j=1}^{p}r_{j}!}
\sup_{z\in[0,1]}\left|\sum_{t=1}^{\Floor{nz}}H_{r_{1}}(\eta_{t})\cdots H_{r_{1}}(\eta_{t})\right|\nonumber\\
&\qquad\times\sup_{z\in[0,1], \mathbf{u}\in[0,1]^{p}}\left|J_{r_{1},\cdots, r_{p}}(v_{\Floor{nz}}(\mathbf{u}))
-J_{r_{1},\cdots, r_{p}}(\mathbf{u})\right|,
\end{align}
where $\left|\sum_{t=1}^{\Floor{nz}}H_{r_{1}, \cdots, r_{p}}(\eta_{t})\right|$ is bounded as we shown in (\ref{3.10}).
By the mean value theorem, we get
\begin{align*}
&\sup_{z\in[0,1], \mathbf{u}\in[0,1]^{p}}\left|J_{r_{1},\cdots, r_{p}}(v_{\Floor{nz}}(\mathbf{u}))
-J_{r_{1},\cdots, r_{p}}(\mathbf{u})\right|\\
&\quad=\sup_{z\in[0,1], \mathbf{u}\in[0,1]^{p}}|v_{\Floor{nz}}(\mathbf{u})-\mathbf{u}||J_{r_{1},\cdots, r_{p}}^{'}
(\widetilde{v}_{\Floor{nz}}(\mathbf{u}))|,
\end{align*}
where $\sup_{z\in[0,1], \mathbf{u}\in[0,1]^{p}}|v_{\Floor{nz}}(\mathbf{u})-\mathbf{u}|\rightarrow 0$ and 
$|\widetilde{v}_{\Floor{nz}}(\mathbf{u})-\mathbf{u}|\leq |v_{\Floor{nz}}(\mathbf{u})-\mathbf{u}|$ and 
$|J_{r_{1},\cdots, r_{p}}^{'}(\mathbf{u})|$ is uniformly bounded. 
Therefore, (\ref{4.9}) converges to zero in distribution. 

For the second term in the right hand side of (\ref{4.8}),
we set $\mathcal{A}(\delta)=\mathbf{u}+\delta\{v_{n}(\mathbf{u})-\mathbf{u}\}$ and $f(\delta)=C(\mathcal{A}(\delta))$ 
for fixed $\mathbf{u}\in[0,1]^{p}$. If $\mathbf{u}\in(0, 1]^{p}$, then $v_{n}(\mathbf{u})\in (0, 1)^{p}$ and $\mathcal{A}(\delta)
\in (0, 1)^{p}$ for all $\delta\in (0, 1]$. The Condition \ref{Condition-4.4} yields the function $f$ is continuous on $[0,1]$
and continuously differentiable on $(0, 1)$. By the mean value theorem, we can write
\begin{align}\label{4.10}
\frac{\Floor{nz}}{q_{n}(r_{1}\cdots r_{p})}\{C(v_{\Floor{nz}}(\mathbf{u}))-C(\mathbf{u})\}&=\frac{\Floor{nz}}{q_{nj}}
\sum_{j=1}^{p}\dot{C}_{j}(\mathcal{A}(\delta))\{D_{\Floor{nz},j}^{-1}(u_{j})-u_{j}\}.
\end{align}
From (\ref{4.10}), we note that if the $j$-th component $u_{ja number of}=0$, then the terms on the right hand side of (\ref{4.10}) is still 
holds true whether $\delta\in (0, 1)$ is well-defined or not, because both side of (\ref{4.10}) are equal to zero. On the other word, 
if $u_{j}=0$ for some $j=1, \cdots, p$, then the $j$-th term is vanished due to $D_{\Floor{nz},j}^{-1}(0)=0$ and the $j$-th 
derivatives $\dot{C}_{j}$ also vanish at this point. By Proposition \ref{Proposition-3.4}, the Bahadur-Kiefer processes will be approximated as
\begin{equation*}
\sup_{z\in[0,1], u_{j}\in[0,1]}\left|\frac{\Floor{nz}}{q_{nj}}\{D_{\Floor{nz},j}^{-1}(u_{j})-u\}+\frac{\Floor{nz}}{q_{nj}}
\{D_{\Floor{nz},j}(u_{j})-u_{j}\}\right|\stackrel{p}\rightarrow 0.
\end{equation*}
as $n\rightarrow\infty$. If the Condition \ref{Condition-4.4} satisfies, then we know that the first 
order partial derivatives $\dot{C}_{j}$ in $[0,1]$ such that
\begin{align*}
&\sup_{\mathbf{u}\in[0,1]^{p}, z\in[0,1]}\left|\frac{\Floor{nz}}{q_{n}(r_{1}\cdots r_{p})}\{C(v_{\Floor{nz}}(\mathbf{u}))-C(\mathbf{u})\}+
\sum_{j=1}^{p}\dot{C}_{j}(\mathbf{u}+\delta\{v_{\Floor{nz}}(\mathbf{u})-\mathbf{u}\})\mathbb{B}_{\Floor{nz},j}(u_{j})\right|
\stackrel{p}\rightarrow 0,
\end{align*}
as $n\rightarrow \infty$. Finally, it remains to show that
\begin{equation*}
\sup_{\mathbf{u}\in[0,1]^{p}, z\in[0,1]}R_{\Floor{nz},j}(\mathbf{u})=\left|\dot{C}_{j}(\mathbf{u}
+\delta\{v_{\Floor{nz}}(\mathbf{u})-\mathbf{u}\})-\dot{C}_{j}(\mathbf{u})\right||D_{\Floor{nz},j}(u_{j})-u_{j}|\stackrel{p}{\rightarrow} 0.
\end{equation*}
as $n\rightarrow \infty$. As in \cite{Segers2012}, for fixed $\varepsilon>0$ and $\omega\in(0, 1/2)$, we can decompose the probability
of $R_{\Floor{nz},j}(\mathbf{u})$ over supremum $\mathbf{u}\in [0, 1]^{p}$ based on the intervals $u_{j}\in [\omega, 1-\omega]$ and $u_{j}\in
[0, \omega)\cup (1-\omega, 1]$ respectively, i.e., 
\begin{align}\label{4.11}
\mathbf{P}\left(\sup_{\mathbf{u}\in[0,1]^{p}, z\in[0,1]}R_{\Floor{nz},j}(\mathbf{u})>\varepsilon\right)&\leq
\mathbf{P}\left(\sup_{\mathbf{u}\in[0,1]^{p}, z\in[0,1], u_{j}\in[\omega, 1-\omega]}R_{\Floor{nz},j}(\mathbf{u})>\varepsilon\right)\nonumber\\
&\quad+\mathbf{P}\left(\sup_{\mathbf{u}\in[0,1]^{p},z\in[0,1], u_{j}\notin[\omega, 1-\omega]}R_{\Floor{nz},j}(\mathbf{u})>\varepsilon\right). 
\end{align}
For the first term on the right-hand side of (\ref{4.11}), we derive that the probability converges to zero. Since
\begin{equation*}
\sup_{z\in[0,1],\mathbf{u}\in[0,1]^{p}}|v_{\Floor{nz}}(\mathbf{u})-\mathbf{u}|\rightarrow 0,
\end{equation*}
almost surely and the partial derivatives $\dot{C}_{j}$ are uniformly continuous on the set
$\{\mathbf{u}\in[0,1]^{p}, u_{j}\in[\omega/2, 1-\omega/2]\}$
and bounded by $0\leq \dot{C}_{j}(\mathcal{A})\leq 1$ and the supremum of the empirical process $\frac{\Floor{nz}}{q_{nj}}
\sup_{u\in[0,1]}|D_{\Floor{nz},j}(u_{j})-u_{j}|$ is bounded in distribution. As for the second term on the right hand side of (\ref{4.11}),
by using the portmanteau lemma we can bound the probability by
\begin{equation*}
\limsup_{n\rightarrow\infty}\mathbf{P}\left(\sup_{z\in[0,1],u_{j}\in[0, \omega)\cup(1-\omega, 1]}|\mathbb{B}_{\Floor{nz},
j}(u_{j})|\geq \varepsilon\right)
\leq \mathbf{P}\left(\sup_{u_{j}\in[0, \omega)\cup(1-\omega, 1]}|\mathbb{B}_{j}(u_{j})|\geq \varepsilon\right).
\end{equation*}
Hence, for any small $\mu>0$, the probability (\ref{4.11}) is
\begin{equation*}
\limsup_{n\rightarrow\infty} \mathbf{P}\left(\sup_{\mathbf{u}\in[0,1]^{p}, z\in[0,1]}R_{\Floor{nz},j}(\mathbf{u})>\varepsilon\right)\leq\mu,
\end{equation*}
which complete the proof of theorem.
\end{proof}

\begin{remark} These convergence results can be extended to Gaussian subordinated processes with covariance 
\begin{equation*}
\gamma_{\eta, j}(k):\sim L_{\eta, j}(k)|k|^{2d_{j}-1}, \quad d_{1}, \cdots, d_{p}\in (0, \frac{1}{2}),
\end{equation*}
where the memory parameters $d_{j}$ are not necessarily all equal. 
\end{remark}

\begin{remark}
$\mathbb{B}_{\Floor{nz}}(\mathbf{u})$ in Lemma \ref{Lemma-4.1} and Lemma \ref{Lemma-4.2} with the case $r=1$ and $r=2$ we can establish
the asymptotic properties of the sequential empirical copula processes $\mathbb{C}_{\Floor{nz}}$ by using the same arguments as in 
Theorem \ref{Theorem-4.5}. More precisely, the asymptotic behavior 
of the empirical copula process $\mathbb{C}_{\Floor{nz}}$ under Condition \ref{Condition-4.4} with lower Hermite rank can also be derived
by the methods of Segers \cite{Segers2012} or functional delte method \cite{Bucher2013}. 
The functional delta methods studied by \cite{Bucher2013} are applied to more general setting. 
\end{remark}

\begin{remark}
If the memory parameters $d_{j}=0$, then the stationary Gaussian sequences $\{\eta_{t,j}\}_{t\in\mathbb{Z}}$ display shor-memory 
behavior and the autocovariance functions are summable. More specifically, if $\sum_{k\in\mathbb{Z}}|\gamma_{\eta, j}(k)|^{r}<\infty$,
then
\begin{equation*}
\mathbb{B}_{n}(\mathbf{u})=\sqrt{n}\{D_{n}(\mathbf{u})-C(\mathbf{u})\}\rightsquigarrow \mathbb{B}_{C}
\end{equation*}
in metric space $D([0,1]^{p})$, where $\mathbb{B}_{C}$ is mean-zero Gaussian field with covariance
\begin{equation*}
\text{Cov}(\mathbb{B}_{C}(\mathbf{u}), \mathbb{B}_{C}(\mathbf{u}))=\sum_{t\in\mathbb{Z}}\text{Cov}\left(
\mathbb{I}(\mathbf{U}_{0}\leq \mathbf{u}), \mathbb{I}(\mathbf{U}_{t}\leq \mathbf{v})\right), \quad\mathbf{u}, \mathbf{v}\in[0,1]^{p}.
\end{equation*}
Furthermore, under Condition \ref{Condition-4.4} and by using the same method, the empirical copula processes $\mathbb{C}_{\Floor{nz}}$
converges weakly to a Gaussian field $\mathbb{G}_{C}$ in $\ell^{\infty}([0,1]^{p})$, which can be expressed as
\begin{equation*}
\mathbb{G}_{C}(\mathbf{u})=\mathbb{B}_{C}(\mathbf{u})-\sum_{j=1}^{p}\dot{C}_{j}(\mathbf{u})\mathbb{B}_{j}(u_{j}), 
\quad\mathbf{u}\in[0,1]^{p}.
\end{equation*}
This can also be achieved by using a functional delta method introduced by \cite{Bucher2013}.
\end{remark}

\section{Conclusion}

\hspace{0.5cm}
We mainly focus on the problem of nonparametric estimation for copula-based time series models under long-range dependence.
The multivariate copula-based time series are subordinated by nonlinear transformation of Gaussian processes, where the observable 
time series exhibit long-memory behavior. We establish limit theorems for the sequential empirical copula processes
$\mathbb{C}_{\Floor{nz}}$ in the context of long-memory. This is essential step to testing problems and it plays a central role
in many applications of copulas.

On the other hand, there are many open problems for future research including semiparametric estimations for copula paramters under Gaussian
subordinated long-memory processes and linear moving avarage process (linear long-memory process), other copula families and 
goodness-of-fit tests ect.\\

{\bf Acknowledgements}\\

I would like to thank Prof. Jan Beran give me a helpful idea and necessary advice to finish the paper.

\bibliographystyle{authordate3}
\bibliography{lit.bib}

\end{document}